\documentclass[11pt]{amsart}
\usepackage{amsmath,amssymb,amsthm,a4wide}
\usepackage{graphicx, color}
\usepackage{comment}

\newtheorem{theorem}{Theorem}[section]    
\newtheorem{lemma}[theorem]{Lemma}          
\newtheorem{proposition}[theorem]{Proposition}

\newtheorem{corollary}[theorem]{Corollary} 

\theoremstyle{definition}
\newtheorem{definition}[theorem]{Definition}
\newtheorem{remark}[theorem]{Remark}  
\newtheorem{observation}[theorem]{Observation}

\newtheorem{conjecture}{Conjecture} 
\newtheorem{example}[theorem]{Example}    
\newcommand{\Z}{\mathbb{Z}}

\newcommand{\Q}{\mathbb{Q}}
\newcommand{\F}{\mathcal F_{ob}}

\newcommand{\mK}{\mathcal{K}}
\newcommand{\mT}{\mathcal{T}}

\newcommand{\sgn}{{\tt sgn}}

\numberwithin{equation}{section}

\title[The defect of Bennequin-Eliashberg inequality and Bennequin surfaces]{
The defect of Bennequin-Eliashberg inequality and Bennequin surfaces}

\author{Tetsuya Ito}
\address{Department of Mathematics, Graduate School of Science, Osaka University \\ 1-1 Machikaneyama Toyonaka, Osaka 560-0043, JAPAN}
\email{tetito@math.sci.osaka-u.ac.jp}
\urladdr{http://www.math.sci.osaka-u.ac.jp/~tetito/}
\author{Keiko Kawamuro}
\address{Department of Mathematics,   
The University of Iowa, Iowa City, IA 52242, USA}
\email{keiko-kawamuro@uiowa.edu}
\date{\today} 


\begin{document}

\begin{abstract}
For a null-homologous transverse link $\mT$ in a general contact manifold with an open book, we explore strongly quasipositive braids and Bennequin surfaces.
We define the defect $\delta(\mT)$ of the Bennequin-Eliashberg inequality. 

We study relations between $\delta(\mT)$ and minimal genus Bennequin surfaces of $\mT$.  
In particular, in the disk open book case, under some large fractional Dehn twist coefficient assumption, we show that $\delta(\mT)=N$ if and only if $\mT$ is the boundary of a Bennequin surface with exactly $N$ negatively twisted bands. That is, the Bennequin inequality is sharp if and only if it is the closure of a strongly quasipositive braid.
\end{abstract}

\maketitle

\section{Introduction}\label{sec:1}

Let $B_{n}$ be the $n$-strand braid group with the standard generators $\sigma_1,\ldots, \sigma_{n-1}$. For $1 \leq i < j \leq n$, let $\sigma_{i,j}$ be the $n$-braid given by 
\[\sigma_{i,j}= (\sigma_{j-1} \sigma_{j-2} \cdots \sigma_{i+1}) \sigma_i (\sigma_{j-1} \sigma_{j-2} \cdots \sigma_{i+1})^{-1} 
\] 
In particular, $\sigma_{i, i+1}=\sigma_i$. 
The braid $\sigma_{i,j}$ (resp. $\sigma_{i,j}^{-1}$) can be understood as the boundary of a positively (resp. negatively) twisted band attached to the $i$-th and the $j$-th strands (see Figure \ref{fig:bennequin}). The elements in the set $\{\sigma_{i,j}\}_{1\leq i<j\leq n }$ are called the \emph{band generators}. 
\begin{figure}[htbp]
\begin{center}
\includegraphics*[width=100mm]{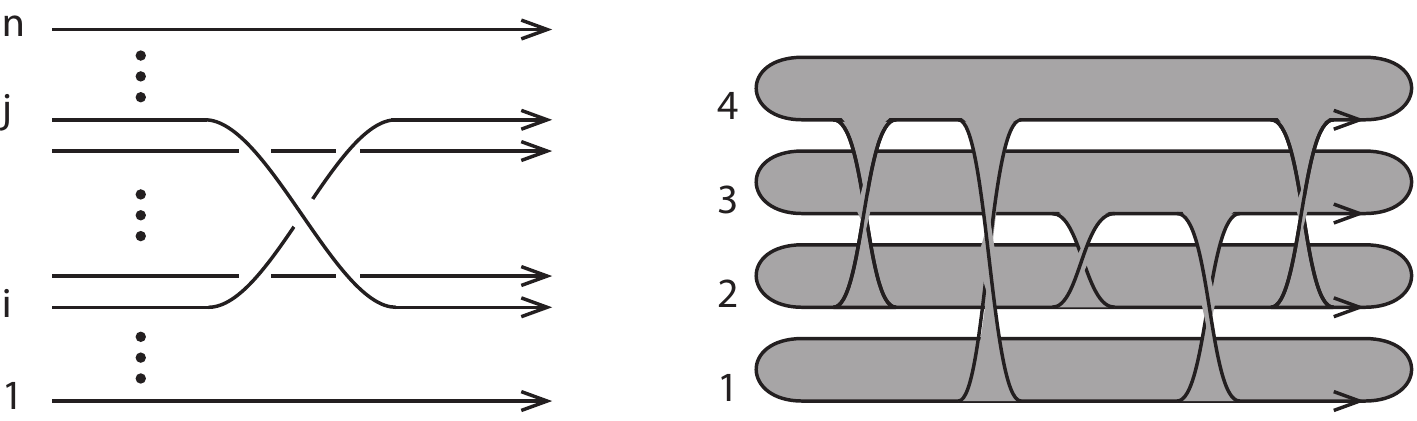}
\caption{(Left) The band generator $\sigma_{i,j}$. (Right) The Bennequin surface associated to 
$\sigma_{2,4}^{-1} \sigma_{1,4} \sigma_{2,3}^{-1} \sigma_{1,3} \sigma_{2,4}^{-1}$.}
\label{fig:bennequin}
\end{center}
\end{figure}

Band generators appear in many papers in the literature.  
The work of Bennequin in \cite{be} identifies braid words in  band generators and transverse knots and links in the standard tight contact $3$-sphere $(S^3, \xi_{std})$.

Rudolph uses band generators in a series of works including \cite{Ru83} where he develops and popularizes the concepts of quasipositive and strongly quasipositive knots and links. 
See also Rudolph's survey article \cite{R}.

Using band generators, Xu in \cite{X} gives a new presentation of $B_3$ and a new solution to the conjugacy problem in $B_3$. 
Birman, Ko and Lee in \cite{bkl} generalize the results of Xu to $B_n$. 
From the modern viewpoint, their work can be understood that the band generators give rise to a Garside structure, which is a certain combinatorial structure allowing us to solve various decision problems like the word and conjugacy problem (see \cite{ddgkm}). Today the Garside structure defined by band generators is called the \emph{dual Garside structure} on $B_n$.

\noindent
{\bf Conventions: }
In this paper, unless otherwise stated, we assume the following: 
\begin{itemize}
\item
Every contact structure is co-oriented.
\item
Every braid word $w \in B_n$ is written in the band generators $\sigma_{i,j}$, rather than in the standard Artin generators $\sigma_1,\dots,\sigma_{n-1}$. 
\item
By ``a link'' we mean an oriented, null-homologous knot or a link. 
\item By ``a transverse link'' we mean a transverse knot or a transverse link which is null-homologous and oriented so that it is positively transverse to the contact planes.
\end{itemize}

Let $\mT$ be a transverse link in $(S^3, \xi_{std})$. We say that a word $w$ in the band generators $\sigma_{i,j}$ is a \emph{braid word representative} of $\mT$ if the closure of the $n$-braid $w$ is $\mT$. 
For a braid word representative $w$ of $\mT$, starting with $n$ disjoint disks and attaching a twisted band for each $\sigma_{i,j}^{\pm 1}$ in the word $w$ we get a Seifert surface $F=F_w$ of $\mT$, which we call the \emph{Bennequin surface} associated to $w$ (see Figure \ref{fig:bennequin}).

A Bennequin surface is defined by Birman and Menasco in \cite[p.71]{bm2} for a {\em topological} link which generalizes   Bennequin's Markov surface \cite{be} (every Bennequin surface is a Markov surface, but there are Markov surfaces which are not Bennequin surfaces \cite[p.73]{bm2}), where they require one more additional condition that $F$ has the {\em maximal} Euler characteristic among all Seifert surfaces.
However, in this paper, $F_w$ may not necessarily realize the maximal Euler characteristic.

A braid $K \in B_{n}$ is called \emph{strongly quasipositive} \cite{R} if $K$ admits a word representative $w$ such that its associated Bennequin surface $F_w$ has no negatively twisted bands. 
That is, $w$ is a product of positive band generators. 
Using the dual Garside structure on $B_n$ with the band generators, one can check whether a given braid $K$ is conjugate to a strongly quasipositive braid or not \cite{bkl}.

Bennequin in \cite{be} shows that for a braid word representative $w$ of $\mT$, the self-linking number $sl(\mT)$ is given by the formula
\begin{equation}
\label{eqn:BEformula}
 sl(\mT)=-n(w)+\exp(w)
\end{equation}
where $n(w)$ and $\exp(w)$ denote the number of braid strands and the exponent sum of $w$. 
He also proves a fundamental inequality called the \emph{Bennequin inequality}
\begin{equation}
\label{eqn:BEinequality}
 sl(\mT) \leq -\chi(\mT) := 2g(\mT)-2+|\mT|, 
\end{equation}
where $g(\mT)$ denotes the $3$-genus (of the topological oriented link type) of $\mT$ and $|\mT|$ denotes the number of link components of $\mT$. 
The topological invariant $\chi(\mT)$ is called the {\em Euler characteristic of $\mT$}. We note that in general 
\begin{equation}\label{eqn:Eulerchar}
\chi(\mT) \geq \chi(F_w).
\end{equation}

\begin{definition}\label{def:defect}
To measure how far the Bennequin inequality (\ref{eqn:BEinequality}) is from the equality, we define the \emph{defect of the Bennequin inequality} for a transverse link $\mT$ by
\[ \delta(\mT): =\frac{1}{2}\bigl(-\chi(\mT)-sl(\mT)\bigr). \] 
Note that $\delta(\mT)$ is a non-negative integer. 
\end{definition}

For a braid word representative $w$ of $\mT$, 
Definition~\ref{def:defect},  (\ref{eqn:BEformula}) and (\ref{eqn:Eulerchar}) imply that 
\begin{eqnarray*}   
0 \leq \delta(\mT) 
&=& 
\frac{1}{2}(-\chi(\mT)-sl(\mT)) \\
&\leq &
\frac{1}{2}(-\chi(F_w)- sl(\mT)) \\
&=& \mbox{the number of negatively twisted bands in } F_w. 
\end{eqnarray*}
Therefore, we observe the following:

\begin{observation}
\label{observation:BE}
The genus of the Bennequin surface $F_{w}$ is equal to $g(\mT)$ 
if and only if the number of negatively twisted bands of $F_{w}$ is equal to $\delta(\mT)$.
In particular, for a strongly quasipositive braid word $w$, its Bennequin surface $F_{w}$ gives a minimum genus Seifert surface of $\mT$ and the Bennequin inequality is sharp, i.e. $\delta(\mT)=0$. 
\end{observation}

Related to Observation \ref{observation:BE} we conjecture the following:

\begin{conjecture}
\label{conjecture:generalSQP} {\em
Every transverse link $\mT$ in $(S^3, \xi_{std})$ is represented by a braid word $w$ whose Bennequin surface $F_w$ contains $\delta(\mT)$ negative bands. 
Equivalently, due to Observation~\ref{observation:BE}, every $\mT$ bounds a Bennequin surface of genus $g(\mT)$. }
\end{conjecture}

In Conjecture~\ref{conjecture:generalSQP}, we do not require that the braid word $w$ realizes the {\em braid index} of the transverse link $\mT$ defined by  
\[ b(\mT):= \min\{n \in \Z_{> 0} \: | \: \mT \mbox{ has an } n\mbox{-braid word representative}\}.\]  
In fact, in \cite{HS} Hirasawa and Stoimenow give an example $\mT$  of $b(\mT)=4$ represented by 
$$\mT = \sigma_{1,2} (\sigma_{2,4})^2 (\sigma_{1,2})^{-1} \sigma_{1,3} \sigma_{1,2} (\sigma_{2,4})^{-1} (\sigma_{1,2})^{-2} (\sigma_{1,3})^{-2}$$ 
(note the sign convention is altered here) and none of whose Bennequin surfaces consisting of four disks and twisted bands have the genus $g(\mT)=3$.

However, studying the open book foliation of the genus $3$ surface depicted in \cite[Fig 2 (b)]{HS} we can verify that one positive stabilization (cf Figure~\ref{fig:stabilization}) of this 4-braid produces a 5-braid representative of $\mT$ that bounds a Bennequin surface of genus $g(\mT)=3$ as sketched in \cite[Fig 2 (d)]{HS}. 
Concerning the the braid index, we give the following stronger version of Conjecture~1.  

\noindent
{\bf Stronger Form of Conjecture 1.}   {\em
Every transverse link $\mT$ in $(S^3, \xi_{std})$ is represented by a braid word $w$ of the braid index at most $b(\mT)+\delta(\mT)$ such that its Bennequin surface $F_w$ contains $\delta(\mT)$ negative bands.  }

Under a condition of large fractional Dehn twist coefficient (FDTC in short, and see the definition in Section~\ref{section:FDTC}), Conjecture~\ref{conjecture:generalSQP} holds as stated in Theorem~\ref{theorem:main2}. 

A special case of Conjecture \ref{conjecture:generalSQP} where $\delta(\mT)=0$ is of our interest.

\begin{conjecture}
\label{conjecture:SQP} {\em
For a transverse link $\mT$ in $(S^3, \xi_{std})$, the Bennequin inequality is sharp if and only if $\mT$ is represented by a strongly quasipositive braid.}

\noindent
\textbf{Stronger Form of Conjecture 2.} {\em
For a transverse link $\mT$ in $(S^3, \xi_{std})$, the Bennequin inequality is sharp if and only if $\mT$ is represented by a strongly quasipositive braid of the braid index $b(\mT)$.}
\end{conjecture}

The statement of Conjecture~\ref{conjecture:SQP} has been existing for more than a decade as a question or as a conjecture among a number of mathematicians, including Etnyre, Hedden \cite[Conjecture 40]{H2}, Rudolph and Van Horn-Morris. 

Under a condition on large FDTC, both Conjecture~\ref{conjecture:SQP} and its stronger form hold as stated in Corollary~\ref{cor:main2}.

Using Hedden's result of {\em topological} fibered knots \cite[Theorem 1.2]{he}, we can immediately show that Conjecture~\ref{conjecture:SQP} holds for fibered transverse knots in $(S^3, \xi_{std})$. More generally, Etnyre and Van Horn-Morris give a characterization of fibered transverse links in general contact 3-manifolds on which the Bennequin-Eliashberg inequality (cf. Theorem~\ref{thmBE}) is sharp \cite[Theorem 1.1]{ev}.

The aim of this paper is to study these conjectures in the setting of general contact 3-manifolds. 

First we recall a fundamental fact repeatedly used in this paper: 
In a general closed oriented contact 3-manifold supported by an open book $(S, \phi)$, every closed braid with respect to $(S, \phi)$ can be seen as a transverse link. 
Conversely, every transverse link can be represented by a closed braid with respect to $(S, \phi)$, which is uniquely determined up to positive braid stabilizations, positive braid destabilizations and braid isotopy (see \cite{be, OS} for the case of disk open book $(D^2, id)$ and \cite{MM, Pav, P2} for general case).

Next, we set up some terminologies.

\begin{definition}
Let $\mT$ be a null-homologous transverse link in a contact 3-manifold $(M,\xi)$. We say that $\alpha \in H_{2}(M,\mT;\Z)$ is a \emph{Seifert surface class} if $\alpha =[F]$ for some Seifert surface $F$ of $\mT$. This is equivalent to $\alpha \in \partial^{-1}([\mT])$, where $[\mT] \in H_{1}(\mT;\Z)$ is the fundamental class of $\mT \cong S^{1}\cup \cdots \cup S^{1}$ and $\partial: H_{2}(M,\mT;\Z) \rightarrow H_{1}(\mT;\Z)$ is the boundary homomorphism of the long exact sequence of the pair $(M,\mT)$. 
Let $sl(\mT,\alpha)$ denote the self-linking number of $\mT$ with respect to $\alpha$.
We say that a Seifert surface $F$ of $\mT$ is an \emph{$\alpha$-Seifert surface} if $[F]=\alpha \in H_{2}(M,\mT;\Z)$. 
\end{definition}

\begin{definition}

Let $g(F)$ be the genus of $F$ and $\chi(F)$ be the Euler characteristic of $F$. 
We define the {\em genus} and the {\em Euler characteristic of $\mT$ with respect to $\alpha$} by 
\[ g(\mT,\alpha) := \min \{g(F) \: | \: F \mbox{ is an } \alpha\mbox{-Seifert surface of } \mT \},\]
$$
\chi(\mT, \alpha) := \max \{\chi(F)\: | \: F \mbox{ is an } \alpha\mbox{-Seifert surface of } \mT \}.
$$
We have $\chi(\mT,\alpha)=2-2g(\mT,\alpha)-|\mT|$,  where $|\mT|$ denotes the number of link components of $\mT$.

\end{definition}

We recall a theorem of Eliashberg.
\begin{theorem}[The Bennequin-Eliashberg inequality \cite{E}]\label{thmBE}
The contact manifold $(M,\xi)$ is tight if and only if 
for any null-homologous transverse link $\mT$ and its Seifert class $\alpha$ we have $$sl(\mT, \alpha)\leq - \chi(\mT, \alpha).$$

For an overtwisted contact manifold $(M,\xi)$, 
the same inequality holds for any null-homologous, non-loose transverse link $\mT$ and its Seifert class $\alpha$. 

\end{theorem}

The second statement is attributed to \'Swiatkowski and a proof can be found in Etnyre's paper \cite[Proposition 1.1]{E2}.

Theorem~\ref{thmBE} guides us to introduce the following invariant.

\begin{definition}
We define the \emph{defect of the Bennequin-Eliashberg inequality with respect to $\alpha$} by
\[ \delta(\mT,\alpha) := \frac{1}{2}\bigl(-\chi(\mT,\alpha)-sl(\mT,\alpha)\bigr).\]
\end{definition}

Note that $\delta(\mT,\alpha)$ is an integer and it can be any negative integer when $\xi$ is overtwisted: To see this, we observe that a transverse push-off of an overtwisted disk gives an transverse unknot $U$ bounding a disk, $D$, with $sl(U,[D])=1$ and $\delta(U,[D]) = -1$. 
Taking some boundary connect sum of $n$ copies of $D$ (with bands each of which contains one positive hyperbolic point as illustrated in Figure~\ref{fig:stabilization} (i)) we get a disk, $D_n$, with $sl(\partial D_n, [D_n])= 2n-1$ and $\delta(\partial D_n,[D_n])=-n$.

In Definition~\ref{definition:BEsurface_geom}, we define 
an {\em $\alpha$-Bennequin surface} with respect to a general open book $(S,\phi)$ that supports a general contact $3$-manifold as an $\alpha$-Seifert surface admitting a disk-band decomposition adapted to the open book $(S,\phi)$. 
We say that a closed braid with respect to $(S, \phi)$ is {\em $\alpha$-strongly quasipositive} if it is the boundary of an $\alpha$-Bennequin surface without negatively twisted bands. 
We say that an $\alpha$-Bennequin surface $F$ is a \emph{minimum genus $\alpha$-Bennequin surface} if $g(F)=g(\mT,\alpha)$.

The definition of $\alpha$-strongly quasipositive has been discussed by Etnyre, Hedden, and Van Horn-Morris since around 2009. 
It was formally introduced by Baykur, Etnyre, Hedden, Kawamuro and Van Horn-Morris in the SQuaRE meeting at the American Institute of Mathematics in July 2015, and is printed in the official SQuaRE report \cite{B}. Later, Ito independently came up with the same definition. 
Hayden also has the same definition \cite[Definition 3.3]{Hayden}.

As we will see in Lemma~\ref{lemma:deltacondition}, if $\mT$ bounds a minimum genus $\alpha$-Bennequin surface, then $\delta(\mT,\alpha) \geq 0$. We expect that the converse is true:

\begin{conjecture}
\label{conjecture:minimumBEsurface} {\em
Let $(S,\phi)$ be an open book decomposition supporting a contact 3-manifold $(M,\xi)$.
Let $\mT$ be a null-homologous transverse link in $(M,\xi)$ with a Seifert surface class $\alpha \in H_{2}(M,\mT;\Z)$.
If $\delta(\mT,\alpha) \geq 0$ then $\mT$ bounds a minimum genus $\alpha$-Bennequin surface with respect to $(S,\phi)$.}
\end{conjecture}

We list evidences for Conjecture~\ref{conjecture:minimumBEsurface}. 

First, we show that the minimum genus Bennequin surface always exists, if we forget the contact structure and only consider the \emph{topological} link type of $\mT$ in $M$.  
The statement is proved by Birman and Finkelstein in \cite[Theorem 4.2]{bf} for a special case where $M=S^3$ and $(S, \phi)=(D^2, id)$ the disk open book.

\begin{theorem}[proved in Section~\ref{subsec:4.3}]
\label{theorem:BEsurface_top}
Let $M$ be an oriented, closed 3-manifold with an open book decomposition $(S,\phi)$. 
For every null-homologous topological link type $\mK$ in $M$ and its Seifert surface class $\alpha \in H_{2}(M,\mK;\Z)$, $\mK$ bounds a minimum genus $\alpha$-Bennequin surface with respect to $(S,\phi)$.
\end{theorem}

Second, in Proposition \ref{prop:PreBennequin} we show that under some condition on the fractional Dehn twist coefficient, a transverse link bounds a minimum genus Seifert surface which is almost an $\alpha$-Bennequin surface.

Third, recall Bennequin's \cite[Proposition 3]{be}.  
The following stronger statement in \cite[Theorem 1]{bm2} is proved by Birman and Menasco, in which a subtle gap in \cite{be} concerning pouches is fixed: 
%
{\em Any minimal genus Seifert surface of a closed 3-braid is isotopic to a Bennequin surface with the same boundary.}
%
This statement implies that Conjecture \ref{conjecture:minimumBEsurface} holds for closed 3-braids with respect to the disk open book $(D^{2},id)$.

The following Proposition \ref{prop:newformulation} and Theorem \ref{theorem:Q3to12} motivate us to study Conjecture \ref{conjecture:minimumBEsurface}.

Due to Theorem~\ref{thmBE}, $(M,\xi)$ is tight if and only if $\delta(\mT,\alpha)\geq 0$ for all null-homologous $\mT$ and its Seifert class $\alpha$. 
Thus, if Conjecture \ref{conjecture:minimumBEsurface} is true  then we obtain a new formulation of tightness in terms of $\alpha$-Bennequin surfaces:

\begin{proposition}[proved in Section~\ref{subsec:4.3}] 
\label{prop:newformulation}
Let $(S,\phi)$ be an open book decomposition supporting a contact 3-manifold $(M,\xi)$.  
For every null-homologous transverse link $\mT$ in $(M,\xi)$ and  its Seifert surface class $\alpha$, we suppose that 
the link $\mT$ bounds a minimum genus $\alpha$-Bennequin surface with respect to $(S,\phi)$. 
Then $(M,\xi)$ is tight.

The converse of the above statement is true if Conjecture~\ref{conjecture:minimumBEsurface} is true. 
Namely, if Conjecture~\ref{conjecture:minimumBEsurface} is true and $(M,\xi)$ is tight, then 
for every null-homologous transverse link $\mT$ in $(M,\xi)$ and  every Seifert surface class $\alpha\in H_2(M,\mT;\Z)$, 
the link $\mT$ bounds a minimum genus $\alpha$-Bennequin surface with respect to $(S,\phi)$.
\end{proposition}

In the setting of general open books, 
Conjectures \ref{conjecture:generalSQP} and \ref{conjecture:SQP} can be extended to Conjectures~ \ref{conjecture:generalSQP}$'$ and \ref{conjecture:SQP}$'$, respectively, as below:   
Let $(S,\phi)$ be an open book decomposition supporting a contact 3-manifold $(M,\xi)$. 
Let $\mT$ be a null-homologous transverse link in $(M, \xi)$ and $\alpha \in H_2(M, \mT;\Z)$ be a Seifert surface class.   

\noindent
{\bf Conjecture \ref{conjecture:generalSQP}$'$.} {\em 
If $\delta(\mT, \alpha)\geq 0$ then $\mT$ bounds an $\alpha$-Bennequin surface with $\delta(\mT, \alpha)$ negative bands with respect to $(S, \phi)$. }

\noindent
{\bf Conjecture \ref{conjecture:SQP}$'$.} {\em
If $\delta(\mT, \alpha)=0$ (in which case we say that the Bennequin-Eliashberg inequality is {\em sharp} on $(\mT, \alpha)$) then $\mT$ is represented by an $\alpha$-strongly quasipositive braid with respect to $(S, \phi)$. }

Conjecture~\ref{conjecture:SQP}$'$ is raised as a question in the SQuaRE report \cite{B}. It is also stated in \cite{B} that a strongly quasipositive link bounds a minimal genus Bennequin surface. 

We remark that for a general open book, the counterpart of the stronger form of Conjecture~\ref{conjecture:SQP} does not hold.  
In Example \ref{example2}, with a fixed open book $(S, \phi)$ of a contact manifold $(M, \xi)$, we give an example of transverse knot $\mT$ in $(M,\xi)$ with $\delta(\mT)=0$ which does bound a minimum genus Bennequin surface but any braid representative of $\mT$ realizing the minimum braid index with respect to $(S, \phi)$ cannot bound minimum genus Bennequin surfaces.

\begin{theorem}[proved in Section~\ref{subsec:4.3}] 
\label{theorem:Q3to12}
If Conjecture \ref{conjecture:minimumBEsurface} is true then Conjectures \ref{conjecture:generalSQP}$'$ and \ref{conjecture:SQP}$'$ are true. 
\end{theorem}

Theorem~\ref{theorem:Q3to12} and the above mentioned 3-braid result yield the following. 

\begin{corollary}\label{cor:3braids}
Let $\mT$ be a transverse link in $(S^3, \xi_{std})$ of the braid index $b(\mT)=3$ with respect to the disk open book $(D^{2},id)$. 
Then  
$\mT$ bounds a minimal genus Bennequin surface that consists of $\delta(\mT)$ negatively twisted bands, a number of positively twisted bands, and three disks.  

In particular, the Bennequin inequality is sharp on $\mT$ if and only if the 3-braid is (braid isotopic to) a strongly quasipositive braid. 
\end{corollary}

\subsection{Main results}

Our first main result Theorem~\ref{theorem:main1} confirms Conjecture \ref{conjecture:SQP}$'$ under some assumptions. 
Let $(S, \phi)$ be an open book and $C$ be a connected component of the binding of $(S, \phi)$, which we will call a \emph{binding component}. 
Let $K$ be a closed braid with respect to $(S, \phi)$ and $c(\phi,K,C)$ be the fractional Dehn twist coefficient (FDTC) of the closed braid $K$ with respect to the binding component $C$  (see Section \ref{section:FDTC} for the definition).

\begin{theorem}[proved in Section~\ref{sec:5}]
\label{theorem:main1}
Let $(S,\phi)$ be an open book decomposition supporting a contact 3-manifold $(M,\xi)$ and $\mT$ be a null-homologous transverse link in $(M,\xi)$ with a Seifert surface class $\alpha \in H_{2}(M,\mT;\Z)$. 
Assume the following: 
\begin{enumerate} 
\item[(i)] $S$ is planar.
\item[(ii)] $M$ does not contain a non-separating 2-sphere (i.e., $M$ does not contain an $S^{1}\times S^{2}$ in its connected summands). 
\item[(iii)] $\mT$ has a closed braid representative $K$ with respect to $(S, \phi)$ which bounds an $\alpha$-Seifert surface $F$ such that: 
\begin{enumerate}
\item[(iii-a)] $g(F)=g(\mT,\alpha)$.
\item[(iii-b)] Among all the binding components of $(S, \phi)$ only $C$ intersects $F$. 
\item[(iii-c)] $c(\phi,K,C)>1$.
\end{enumerate}
\end{enumerate}
Then $\delta(\mT,\alpha)=0$ if and only if $K$ is $\alpha$-strongly quasipositive with respect to $(S, \phi)$. In particular, $\delta(\mT,\alpha)=0$ if and only if $\mT$ is represented by an $\alpha$-strongly quasipositive braid. 
\end{theorem}

If we drop the assumption (i) or (iii-c), as shown in Examples \ref{ex1} and \ref{ex2}, $K$ may not be $\alpha$-strongly quasipositive.
However, we note that this does not mean failure of Conjecture \ref{conjecture:SQP}$'$ since some positive stabilizations of $K$ has a good chance to be $\alpha$-strongly quasipositive.

Our second main result Theorem~\ref{theorem:main2} (and Corollary~\ref{cor:main2})  shows that Conjecture \ref{conjecture:minimumBEsurface} holds for the disk open book $(D^{2},id)$ under an assumption of large FDTC.

\begin{theorem}[proved in Section~\ref{sec:5}]
\label{theorem:main2}
Let $\mT$ be a transverse link in $(S^3, \xi_{std})$. 
Consider the disk open book $(D^{2},id)$ that supports $(S^3, \xi_{std})$.  
If $\mT$ admits a closed braid representative $K$ such that
$$ c(id,K,\partial D^{2}) > \frac{\delta(\mT)}{2}+1 $$
then $\mT$ $($in fact, $K$ itself or $K$ with one positive stabilization$)$ bounds a minimum genus Bennequin surface with respect to $(D^{2},id)$.

Moreover, if $\delta(\mT) =0$ and $c(id,K,\partial D^{2}) >1$ then $K$ is a strongly quasipositive braid.
\end{theorem}

\begin{corollary}
\label{cor:main2}
Let $\mT$ be a transverse link in $(S^3, \xi_{std})$. 
Assume that $\mT$ is represented by a closed braid $K$ with $c(id,K, \partial D^{2}) >1$ and realizing the braid index $b(\mT)$. Then the Bennequin inequality for $\mT$ is sharp if and only if $\mT$ is represented by a strongly quasipositive braid of braid index $b(\mT)$ with respect to $(D^2, id)$.
(Namely, the stronger form of Conjecture \ref{conjecture:SQP} holds.)
\end{corollary}

In Example \ref{example:main} we present examples of braids satisfying conditions in Theorem~\ref{theorem:main2} and Corollary~\ref{cor:main2}. In particular, our example contains many non-fibered knots which shows independency of our results from Hedden's \cite{he}.

Although it looks restrictive, the large FDTC assumption is satisfied by {\em almost all} braids: Indeed, given a random $n$-braid $\beta$ and a number $C$, the probability that $|c(id, \widehat{\beta},\partial D^{2})|\leq C$ is zero (see \cite{ma,it} for the precise meaning of  ``random'').

\section{The FDTC for closed braids in open books}
\label{section:FDTC}

In this section we review closed braids in open books and the FDTC for closed braids.

Let $S$ be an oriented compact surface with non-empty boundary, and $P=\{p_1,\ldots,p_n\}$ be a (possibly empty) finite set of points in the interior of $S$.
Let $MCG(S,P)$ (denoted by $MCG(S)$ if $P$ is empty) be the mapping class group of the punctured surface $S \setminus P$; that is, the group of isotopy classes of orientation-preserving homeomorphisms on $S$, fixing $\partial S$ point-wise and fixing $P$ set-wise.

With respect to a connected boundary component $C$ of $S$, the {\em fractional Dehn twist coefficient} ({\em FDTC}) of $\phi \in MCG(S, P)$, defined in \cite{hkm}, is a rational number $c(\phi,C)$ and measures to how much the mapping class $\phi$ twists the surface near the boundary $C$.

Let $(M,\xi)$ be a closed oriented contact 3-manifold supported by an open book decomposition $(S,\phi)$. 
(See \cite{Et} for the meaning of ``supported''.)
Let $B\subset M$ be the binding of the open book and $\pi: M\setminus B \rightarrow S^{1}= [0,1]\slash (0\sim 1)$ be the associated fibration. 
For $t \in S^{1}$ we denote the closure of the fiber $\pi^{-1}(\{t\})$ by $S_t$ and call it a \emph{page}. 
The topological type of $S_t$ is $S$ and the orientation os $S_t$ induces the orientation of the binding $B$. 
Since $M\setminus B$ is diffeomorphic to $S \times [0,1] / (x, 1)\sim(\phi(x), 0)$ we may denote $S \times\{t\}$ by the same notation $S_t$. 
Let $p: M\setminus B \to S$ be the projection map such that $p|_{S_t}: S_t \to S$ gives a diffeomorphism. 
 
A \emph{closed braid} $K$ with respect to $(S,\phi)$ is an oriented link in $M \setminus B$ which is positively transverse to each page. Two closed braids are called \emph{braid isotopic} if they are isotopic through closed braids. The number of intersection points of  $K$ and the page $S_t$ is denoted by $n(K)$ and called the \emph{braid index} of $K$ with respect to $(S, \phi)$.

Let $B_{n}(S)$ be the $n$-stranded surface braid group for $S$ (see \cite[p.244]{FM} for the definition).
Cutting $M \setminus B$ along the page $S_0$ we get a cylinder ${\rm int}(S) \times (0,1)$ and the closed braid $K$ gives rise to a surface braid $\beta_{K} \in B_{n}(S)$ with $n=n(K)$ strands.

The converse direction; namely, obtaining a closed braid from a surface braid $\beta \in B_n(S)$, requires more care. Recall the generalized Birman exact sequence \cite[Theorem 9.1]{FM} 
\begin{equation}\label{BirmanExSq}
1 \longrightarrow B_{n}(S) \stackrel{i}{\longrightarrow} MCG(S,P)\stackrel{f}{\longrightarrow}  MCG(S) \longrightarrow 1
\end{equation}
where $i$ is the push map and $f$ is the forgetful map. 
Since $f$ is not injective we have various ways to construct a closed braid $K$ from a given braid $\beta \in B_{n}(S)$. 

We recall the definition in \cite{ik2} of the FDTC $c(\phi,K,C)$ of $K$ as follows.

Suppose that the mapping class $\phi \in MCG(S)$ is represented by a homeomorphism $f \in {\rm Homeo}^+(S, \partial S)$.  
For a connected boundary component $C$ of $S$, let us choose a collar neighborhood $\nu(C) \subset S$ of $C$. 
We may assume that $f$ fixes $\nu(C)$ point-wise. 
We say that a closed braid $K$ is \emph{based on $C$} if $P= p(K \cap S_{0}) \subset \nu(C)$.
We may isotop $K$ through closed braids so that $K$ is based on $C$.  
Since $f|_{\nu(C)} = id$, the puncture set $P$ is pointwise fixed by $f$; thus, we may view $f$ as an element of ${\rm Homeo}^+(S, P, \partial S)$. In order to distinguish $f\in {\rm Homeo}^+(S, \partial S)$ and $f\in {\rm Homeo}^+(S, P, \partial S)$ we denote the latter by $j(f)$. 
The map $j$ induces a homomorphism
$$j_{*}:MCG(S) \rightarrow MCG(S,P)$$
which satisfies $j_*(\phi) = [j(f)]$.

\begin{definition}\label{def:distinguished-monodromy}
Let $K$ be a closed braid with respect to $(S,\phi)$ and based on $C$. The \emph{distinguished monodromy} of the closed braid $K$ with respect to $C$ is the mapping class 
\[ \phi_K = i(\beta_{K}) \circ j_{*}(\phi) \in MCG(S,P).\]
Here $i$ denotes the push map in the generalized Birman exact sequence. 
The FDTC of a closed braid $K$ with respect to $C$ is defined by
\[ c(\phi,K,C) := c(\phi_K,C). \]
\end{definition}

The FDTC $c(\phi, K, C)$ is well-defined; namely, if braids $K_1$ and $K_2$ are braid isotopic and $p(K_i \cap S_0) \subset \nu(C)$ for both $i=1, 2$ then $c(\phi, K_1, C)=c(\phi, K_2, C)$. In fact, a stronger statement can be found in \cite[Proposition 2.4]{ik-QRV}. 

\begin{remark} 
\label{remark:dualGarside}
Due to the dual Garside structure of the braid group $B_n$  coming from the band generators $\sigma_{i,j}$, \cite[Theorem 3.10]{bkl} states that each $\beta \in B_{n}$ can be represented by the unique left canonical normal form  
$$N(\beta) = \delta^{N}x_1 \cdots x_k,$$
where 
$$
\delta = \sigma_{n-1,n}\sigma_{n-2,n-1}\cdots \sigma_{1,2} 
$$
and $x_1,\ldots,x_k$ are certain strongly quasipositive braids called the \emph{dual simple elements}. 
As a homeomorphism of a disk with $n$ evenly distributed punctures along the boundary, $\delta$ rotates the disk by $\frac{2 \pi}{n}$. 
The integer $N$ in the normal form $N(\beta)= \delta^{N}x_1 \cdots x_k$ is called the \emph{infimum} of $\beta$ \cite[p.337]{bkl} and denoted by $\inf(\beta)$. The \emph{infimum}  of an $n$-stranded closed braid $K$ with respect to the disk open book $(D^2, id)$ is defined by 
$$
\inf(K) = \max\{\inf(\beta)\: | \: \beta \in B_n  \mbox{ is a braid whose closure is braid isotopic to } K\}.
$$ 
We observe that $K$ is strongly quasipositive if and only if $\inf(K)\geq 0$.

Although both $\inf(K)$ and $c(id,K,\partial D^{2})$ count the number of twists near the boundary $\partial D^{2}$ in certain ways, in general, there is no direct connection between them. 
For example, for $n\geq 3$ let 
$$\beta=\sigma_{1,2}\sigma_{2,3} \cdots \sigma_{n-1,n}\sigma_{1,n} \in B_{n}$$
(read from left to right) and $K_m$ be the braid closure of $\beta^{m}$ where $m\in\mathbb N$. The following discussion shows that $\inf(K_m)=0$ and $c(id,K_m,\partial D^{2})=m$. 

Let $x_{i}:=\sigma_{[i],[i+1]}$ for every $i\in \mathbb N$, where $1\leq [i]\leq n$ denotes the unique integer that satisfies $[i] \equiv i \pmod n$ and $\sigma_{n,1}:=\sigma_{1,n}$. 
The normal form of $\beta^{m}$ is
\begin{align*}
N(\beta^{m}) & = x_1 x_2 \cdots x_{nm}\\
& = \left((\sigma_{1,2})(\sigma_{2,3})\cdots (\sigma_{n-1,n})(\sigma_{1,n}) \right)^m.
\end{align*}
Therefore, $\inf(\beta^{m})=0$.
Furthermore, $\beta^{m}$ is rigid; that is, conjugation by the 1st factor of the normal form $x_{1}=\sigma_{1,2}$ 
\[ x_1^{-1}N(\beta^{m})x_1=x_2x_3\cdots x_{nm}x_1 \] 
produces a normal form of the braid $\sigma_{1,2}^{-1}\beta^{m}\sigma_{1,2}$. This implies that $\beta^{m}$ attains the maximum infimum among its conjugacy class \cite[Lemma 3.13]{bgg}; hence, $\inf(K_m)=0$.

On the other hand, looking at the image of some properly embedded arc $\gamma$ under the braid $\beta^{m}$ we obtain $T_{\partial D}^{m-1} (\gamma) \geq \beta^m(\gamma) \geq T_{\partial D}^{m} (\gamma)$ for all $m\in \mathbb N$. Thus by \cite[Theorem 4.14]{ik2} and \cite[Proposition 4.10]{ik2} we have 
$1-\frac{1}{m} \leq c(id,K_1,\partial D^{2}) \leq 1$ for all $m\in \mathbb N$. 
This gives that $c(id,K_1,\partial D^{2}) =1$, and it follows that  $c(id,K_m,\partial D^{2})=m$.  
\end{remark}

\section{Summary of results in open book foliations}\label{section:OBFsummary}

In this section, we review properties of open book foliations that are  needed to prove our main theorems.
For details, see \cite{ik1-1,ik2,ik3}.

Let $(S,\phi)$ be an open book decomposition supporting a contact 3-manifold $(M,\xi)$. 
Throughout this section, the open book decomposition $(S, \phi)$ is fixed. 
Let $K$ be a closed braid with respect to $(S, \phi)$ and $F$ be a Seifert surface of $K$. 
With an isotopy fixing $K=\partial F$, \cite[Theorem 2.5]{ik1-1} shows that $F$ can admit a singular foliation 
\[ \F(F):=\left\{ F \cap S_t \ | \ t \in [0, 1] \right\} \] 
induced by the intersection with the pages of the open book
and satisfying the following conditions. 
\begin{description}
\item[($\mathcal F$ i)] 
The binding $B$ pierces $F$ transversely in finitely many points. 
At each $p \in B \cap F$ there exists a disc neighborhood $N_{p} \subset F$ of $p$ on which the foliation $\F(N_p)$ is radial with the node $p$, see Figure~\ref{fig:sign}-(i). We call $p$ an {\em elliptic} point. 
\item[($\mathcal F$ ii)] 
The leaves of $\F(F)$ are transverse to $K=\partial F$. 
\item[($\mathcal F$ iii)] 
All but finitely many pages $S_{t}$ intersect $F$ transversely.
Each exceptional page is tangent to $F$ at a single point that lies  in the interiors of both $F$ and $S_{t}$.
In particular, $\F(F)$ has no saddle-saddle connections.
\item[($\mathcal F$ iv)] 
All the tangent points of $F$ and fibers are of saddle type, see Figure~\ref{fig:sign}-(ii). 
We call them {\em hyperbolic} points.
\end{description}
Such a foliation $\F(F)$ is called an {\em open book foliation} on $F$. 
\begin{figure}[htbp]
\begin{center}
\includegraphics*[bb= 163 551 442 715,width=90mm]{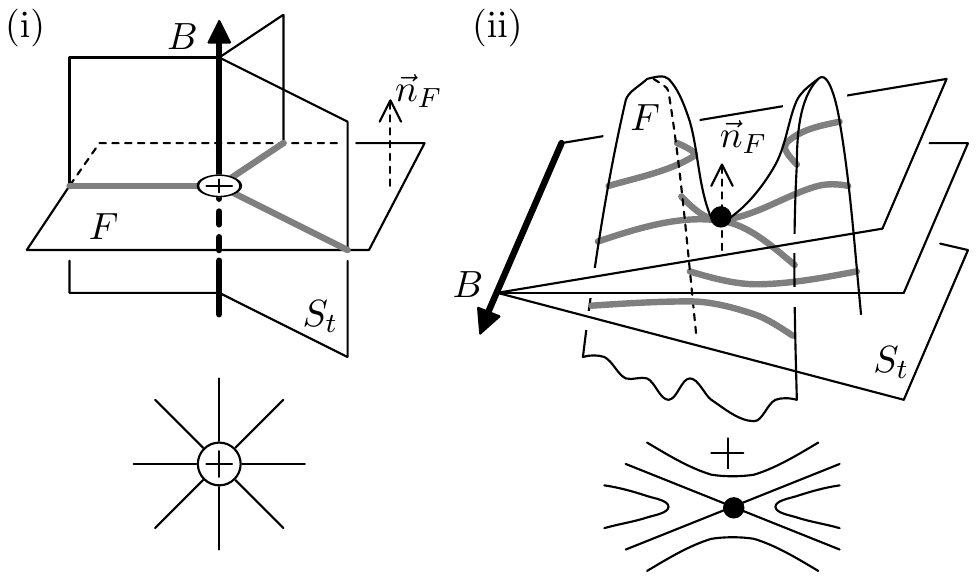}
\caption{ 
(i) A positive elliptic point and (ii) a positive hyperbolic point. 
The dashed arrow $\vec{n}_F \in T_pM$ depicts a positive normal vector to the oriented surface $F$ at a point $p \in F$.}
\label{fig:sign}
\end{center}
\end{figure}

An elliptic point $p$ is {\em positive} (resp. {\em negative}) if the binding $B$ is positively (resp. negatively) transverse to $F$ at $p$.
A hyperbolic point $q \in F \cap S_t$ is {\em positive} (resp. {\em negative}) if the orientation of the tangent plane $T_q(F)$ agrees (resp. disagrees) with the orientation of $T_q(S_t)$.

A leaf of $\F(F)$, a connected component of $F \cap S_t$, is called \emph{regular} if it does not contain a hyperbolic point, and called {\it singular} otherwise.
The regular leaves are classified into the following three types.
\begin{enumerate}
\item[a-arc]: An arc one of whose endpoints lies on $B$ and the other lies on $K$.
\item[b-arc]: An arc whose endpoints both lie on $B$.
\item[c-circle]: A simple closed curve.
\end{enumerate}

The leaves of $\F(F)$ are equipped with orientations as follows (cf. \cite[Definition 2.12]{ik1-1} and \cite[p.80]{OzSt}).  
Take a non-singular point $p$ on a leaf $l$ in a page $S_{t}$.
Let $\vec{n}_F \in T_p S_t \subset T_pM$ be a positive normal vector to the tangent space $T_p F$ and let $v \in T_pl$ be a vector such that $(\vec{n}_F, v)$ gives an oriented bases for the tangent space $T_pS_t$. 
The vector $v$ defines the orientation of the leaf $l$.  
With this orientation of leaves, every  positive (resp. negative) elliptic point becomes a source (resp. sink), and 
the leaves are pointing out of the surface $F$ along the boundary $\partial F$.

According to the types of nearby regular leaves, hyperbolic points are classified into six types: Type $aa$, $ab$, $ac$, $bb$, $bc$ and $cc$.
Each hyperbolic point has a canonical neighborhood as depicted in Figure ~\ref{fig:regions}, which we call a {\em region}. We denote by $\sgn(R)$ the sign of the hyperbolic point contained in the region $R$. 
If $\F(F)$ contains at least one hyperbolic point, then we can decompose $F$ as the union of regions whose interiors are disjoint. We call such a decomposition a \emph{region decomposition}. 

\begin{figure}[htbp]
\begin{center}
\includegraphics*[bb=151 579 446 716,width=90mm]{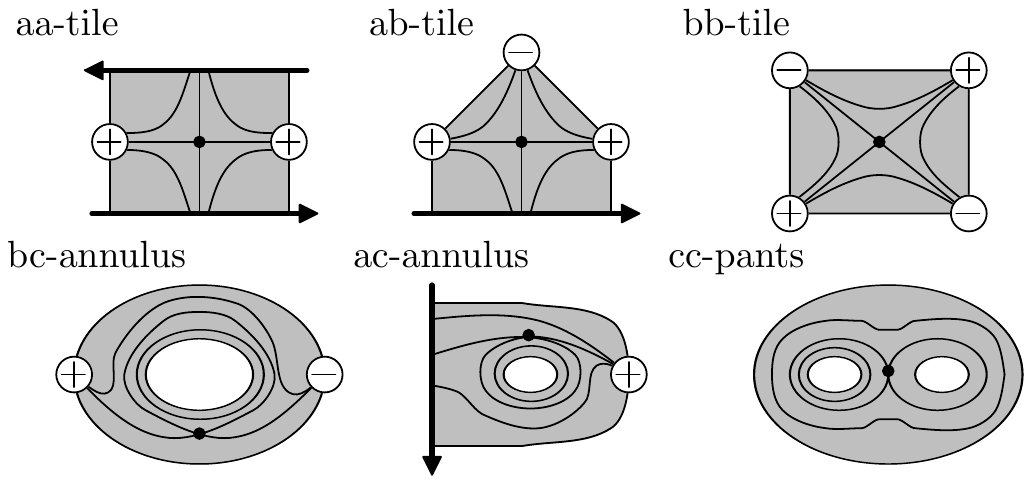}
\caption{Six types of regions.}
\label{fig:regions}
\end{center}
\end{figure}

One can read the Euler characteristic and the self-linking number from the open book foliation.

\begin{lemma}\cite[Proposition 2.11, Proposition 3.2]{ik1-1}
\label{lemma:sl-from-ob}
Let $F$ be a Seifert surface of a transverse link $\mT$ admitting  an open book foliation $\F(F)$. Let $e_{\pm}$ $($resp. $h_{\pm})$ be the number of positive and negative elliptic $($resp. hyperbolic$)$ points of $\F(F)$.
Then the self-linking number has
\[ sl(\mT,[F])= -(e_{+} - e_{-}) + (h_{+}-h_{-}).\]
For the Euler characteristics we have
\[ \chi(\mT, [F]) \geq \chi(F) = (e_{+} + e_{-}) - (h_{+}+h_{-}).\]
Therefore, $\delta(\mT,[F]) \leq h_{-}-e_{-}$. 
In particular, if $g(F)=g(\mT,[F])$ then 
\[ \delta(\mT,[F]) = h_{-}-e_{-}. \]
\end{lemma}

We say that a b-arc $b$ in a page $S_{t}$ is \emph{essential} if $b$ is not boundary-parallel as an arc of the punctured page $S_{t} \setminus (S_{t} \cap K)$. 
We say that an open book foliation $\F(F)$ is \emph{essential} if all the $b$-arcs are essential (cf. \cite[Definition 3.1]{ik2}).
The next theorem shows that an incompressible surface admits an essential open book foliation, after desumming essential spheres that are 2-spheres that do not bound 3-balls:

\begin{theorem}\cite[Theorem 3.2]{ik2}
\label{thm:Thm3.2}
Suppose that $F$ is an incompressible Seifert surface of a closed braid $K$.
Then there exist a Seifert surface $F'$ of $K$ admitting an essential open book foliation and essential spheres $\mathcal{S}_{1},\ldots,\mathcal{S}_{k}$ such that $F$ is isotopic to $F'\# \mathcal{S}_1 \# \cdots \# \mathcal{S}_k$ by an isotopy that fixes $K=\partial F$.
Moreover, if $F$ does not intersect a binding component $C$ then neither does $F'$.
\end{theorem}

Here is a corollary of Theorem~\ref{thm:Thm3.2} which we use later for the proofs of our main results. 

\begin{corollary}
\label{cor:essential}
Assume that $M$ contains no non-separating 2-spheres.
Let $K$ be a closed braid representative of a null-homologous transverse link $\mT$ and $F$ be an incompressible Seifert surface of $K$. 
Then there is an incompressible Seifert surface $F'$ of $K$ with the following properties:
\begin{itemize}
\item $F'$ admits an essential open book foliation.
\item $[F]=[F'] \in H_{2}(M,\mT;\Z)$. 
\item $g(F')=g(F)$.
\item If $F$ does not intersect a binding component $C$ then neither does $F'$. 
\end{itemize}
\end{corollary}

The following theorem gives a connection between essential open book foliations and the FDTC of braids.

\begin{theorem}\cite[Theorem 5.5, Theorem 5.12]{ik2}
\label{lemma:FDTC}
Let $F$ be an incompressible Seifert surface of a closed braid $K$ equipped with an essential open book foliation. 
 Let $v_1,\ldots, v_n$ be negative elliptic points which lie on the same binding component $C$. Let $N$ be the number of negative hyperbolic points that are connected to at least one of $v_1,\ldots,v_n$ by a singular leaf. Then we have
\[ c(\phi,K,C) \leq \frac{N}{n}. \]
\end{theorem}

\section{Generalized Bennequin surfaces}
\label{section:Bennequin_surface}

\subsection{Definition of $\alpha$-Bennequin surfaces} 

In this subsection, we generalize the notion of Bennequin surfaces in $S^3$ with respect to the disk open book $(D^2, id)$ to Bennequin surfaces in a general manifold $M$ with respect to a general open book $(S, \phi)$. 

Let $(S,\phi)$ be an open book. 
We take an annular neighborhood $\nu=\nu(\partial S) \subset S$ of $\partial S$ and fix a homeomorphism 
$$
\nu \approx \underset{|\partial S|}{\sqcup} S^{1} \times [0,1)
$$
so that $\partial S \subset \nu$ is identified with $\underset{|\partial S|}{\sqcup}S^1 \times \{0\}$. Take a set of points $P=\{p_1,\ldots,p_n\}$ so that $P \subset  \underset{|\partial S|}{\sqcup} S^{1} \times \{1/2\}$. Let $\frac{1}{2}\nu:=\underset{|\partial S|}{\sqcup} S^{1}\times [0,\frac{1}{2}) \subset \nu$.
See Figure~\ref{fig:bandtwist}. 
\begin{figure}[htbp]
\begin{center}
\includegraphics*[bb=  198 583 402 714,width=70mm]{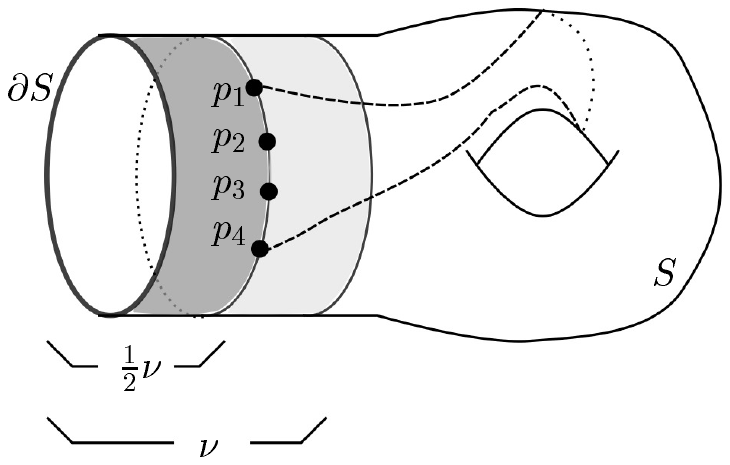}
\caption{Regions $\nu, \frac{1}{2}\nu$ and a properly embedded arc in $S \setminus \frac{1}{2}\nu$ connecting $p_1$ and $p_4$. }
\label{fig:bandtwist}
\end{center}
\end{figure} 

We view $B_{n}(S)$ as a subgroup of $MCG(S,P)$ through the push map $i$ in the generalized Birman exact sequence (\ref{BirmanExSq}). We say that a braid $w \in B_{n}(S)$ is a {\em positive} (resp. {\em negative}) \emph{band-twist} if $w\in MCG(S,P)$ is a positive (resp. negative) half twist about a properly embedded arc  in $S \setminus \frac{1}{2}\nu$ connecting two distinct points in $P$. See Figure \ref{fig:bandtwist} and Figure~\ref{fig:half-twist}. 
\begin{figure}[htbp]
\begin{center}
\includegraphics*[width=100mm]{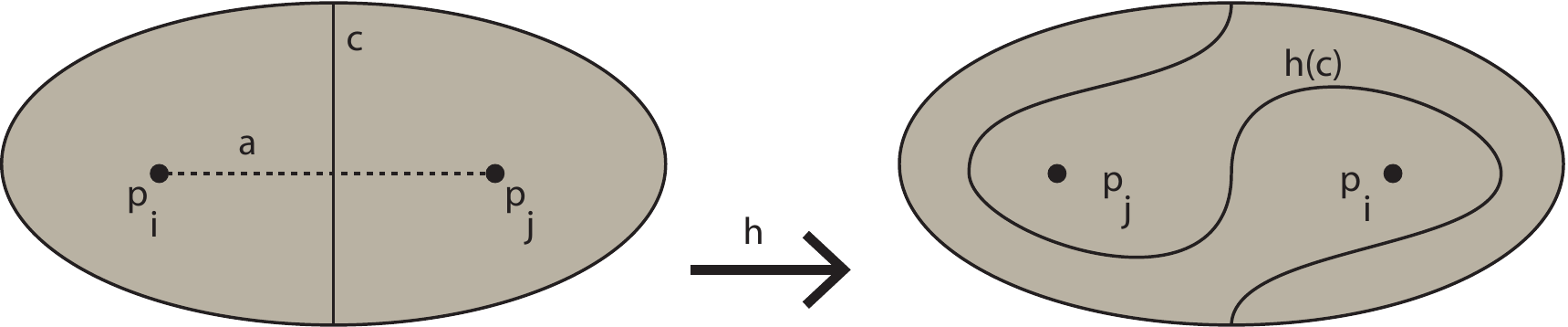}
\caption{A positive half-twist $h$ about an arc $a$ that joins $p_i$ and $p_j$. }
\label{fig:half-twist}
\end{center}
\end{figure}

\begin{definition}
A \emph{band-twist factorization} of a braid $\beta \in B_{n}(S)$ is a factorization of $\beta$ into a word $w_1\cdots w_m$, where each $w_i$ is a band-twist. 
\end{definition}

In the case of $S=D^{2}$, a band-twist factorization is nothing but a factorization using the band generators $\sigma_{i,j}$.

When $g(S)>0$, some braids in $B_{n}(S)$ may not admit  band-twist factorizations: For example, a non-trivial 1-braid in $B_{1}(S)\cong \pi_1(S)$ does not admit a band-twist factorizations.

\begin{definition}[Construction of $F_w$]
\label{definition:BEsurface_alg}
For a closed braid representative $K$ of a transverse link $\mT$ we may isotop $K$ through closed braids so that $p(K \cap S_0)=P$. Let $\beta_{K} \in B_{n}(S)$ be the $n$-braid obtained from $K$ by cutting $M$ along $S_0$. 
Suppose that $\beta_K$ admits a band-twist factorization  $$w=w_1\cdots w_m.$$ 

Take $n$ disjoint meridional disks of the binding $B$ bounded by $$P\times [0,1]/(x,1)\sim(x,0).$$ 
Take a sequence $0< t_1 < \cdots < t_m < 1$. 
For each positive (resp. negative) band-twist $w_i \in MCG(S, P)$, let $\gamma_i$ be a properly embedded arc in $S \setminus \frac{1}{2}\nu$ joining distinct points $x_i$ and $y_i\in P$
such that a positive (resp. negative) half-twist about $\gamma_i$ represents $w_i$. 
For each $i=1,\cdots,m$, we attach a positively twisted band whose core is $\gamma_i \times \{t_i\} \subset S_{t_i}$ to the two meridional disks corresponding to the puncture points $x_i$ and $y_i$, see Figure~\ref{fig:twist-band2}. 

The resulting surface is a Seifert surface of the closed braid $K$ and is denoted by $F_w$. 
\end{definition}  
\begin{figure}[htbp]
\begin{center}
\includegraphics*[width=100mm]{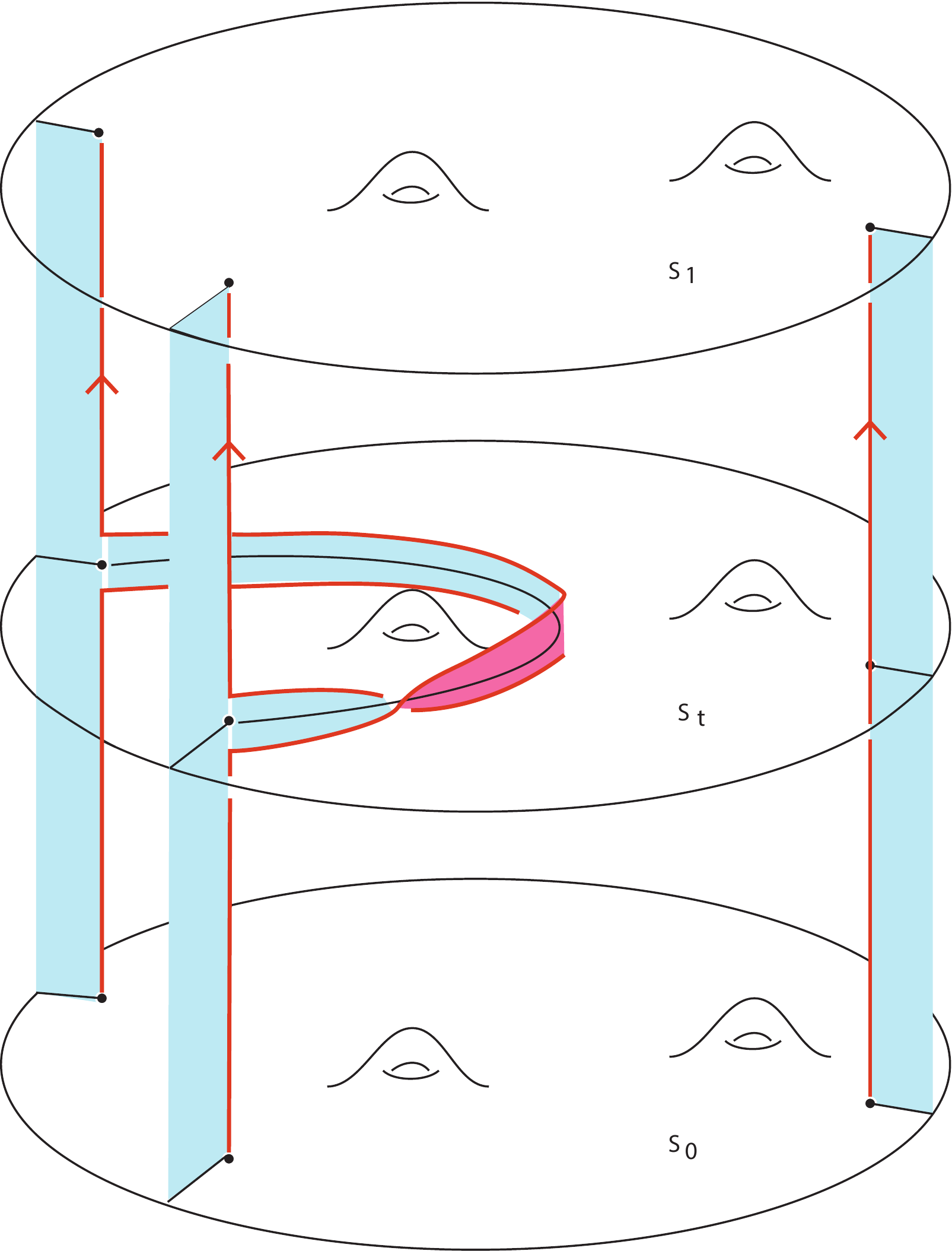}
\caption{Construction of $F_w$. 
The front side is colored blue and the  back side is pink. 
The three vertical rectangle strips become meridional disks in $M_{(S, \phi)}$. A positively twisted band is attached to two meridional disks. The core of the band is $\gamma \times \{t\} \subset S_t$. Orange curves are in braid position.   }
\label{fig:twist-band2}
\end{center}
\end{figure}

Let $(S,\phi)$ be an open book decomposition supporting a contact 3-manifold $(M,\xi)$.
Let $\mT$ be a null-homologous transverse link in $(M,\xi)$ and $\alpha \in H_{2}(M,\mT;\Z)$ be a Seifert surface class. 
Let $K$ be a closed braid representative of $\mT$ with respect to $(S, \phi)$. 

\begin{definition}
\label{definition:BEsurface_geom}
{$ $ } \\
{\bf (1):} An $\alpha$-Seifert surface $F$ of $K$ is called an \emph{$\alpha$-Bennequin surface} of $K$ with respect to $(S,\phi)$ if $F$ admits an open book foliation whose region decomposition consists of only aa-tiles.

\noindent
{\bf (2):} We say that the closed braid $K$ is \emph{$\alpha$-strongly quasipositive} with respect to $(S,\phi)$ if it is the boundary of an $\alpha$-Bennequin surface without negative hyperbolic points. 
(In this case, we also say that $\mT$ is \emph{$\alpha$-strongly quasipositive}.) 
\end{definition}

\begin{remark} 
\begin{itemize}
\item
As noted in Section~\ref{sec:1} the definition of $\alpha$-strongly quasipositive with respect to an open book (Definition \ref{definition:BEsurface_geom} (2)) has been introduced in \cite[Definition 3]{B}. Hayden independently has the same definition \cite[Definition 3.3]{Hayden}.
\item
It is straightforward from Definition \ref{definition:BEsurface_geom} (2) that the Bennequin-Eliashberg inequality is sharp 
on every $\alpha$-strongly quasipositive transverse link, which is also stated in \cite[Corollary 6.3]{Hayden}. 
\item
When $(S, \phi)=(D^2, id)$, the above $\alpha$-Bennequin surface is the same as the Bennequin surface defined by Birman and Menasco in \cite[p.71]{bm2}, and the above $\alpha$-strongly quasipositive is the same as strongly quasipositive defined by Rudolph \cite{Ru83, R}. \end{itemize}
\end{remark}

\begin{proposition}\label{prop:4.5}
The Seifert surface $F_w$ constructed from a band-twist factorization $w$ of a braid $\beta_K$ is an $\alpha$-Bennequin surface, where $\alpha = [F_w]\in H_2(M, K;\Z)$.
\end{proposition}

\begin{proof}
The open book foliation of each meridional disk of $F_w$ contains one positive elliptic point at the intersection with the binding $B$ and a-arcs emanating from the elliptic point. 
The open book foliation of each $\pm$-twisted band contains one singular leaf with $\pm$ hyperbolic point such that its stable separatrix is the arc $\gamma_i \times \{t_i\}$ using the notation in Definition~\ref{definition:BEsurface_alg}. 
\end{proof}

\begin{proposition}
The boundary of every $\alpha$-Bennequin surface with respect to $(S,\phi)$ is a closed braid with respect to $(S, \phi)$ which admits a band-twist factorization. 
\end{proposition}

\begin{proof}
Since an $\alpha$-Bennequin surface $F$ admits an aa-tile decomposition, $F$ is a union of disks each of which is a regular neighborhood of a positive elliptic point $q_k$ of $\F(F)$ and twisted bands each of which is a rectangular neighborhood of a singular leaf containing one hyperbolic point of $\F(F)$.  The sign of each twisted band is equal to the sign of the corresponding hyperbolic point.

Up to isotopy that preserves the topological type of the open book foliation $\F(F)$ we may assume that: 
\begin{itemize}
\item
Each disk is centered at $q_k$ and its boundary 
is described as $p_k \times [0,1]/(p_k,1) \sim (p_k,0)$ for some point $p_k \in \frac{1}{2}\nu$, and
\item
The stable separatrices $\eta_1,\cdots,\eta_m$ of the singular leaves of $\F(F)$ lie on distinct pages $S_{t_i}$ for some $0<t_1 <\cdots < t_m<1$. 
Then the projection $\gamma_i=p(\eta_i)\subset S$ is a properly embedded arc in $S$ joining points, say $q_{k_i}$ and $q_{k'_i} \in\partial S$. 
\end{itemize}
We may further assume that $p_{k_i}, p_{k'_i} \in \gamma_i$ and denote the subarc of $\gamma_i$ joining $p_{k_i}$ and $ p_{k'_i}$ by $\gamma'_i$. 
Let $w_i$ be a band-twist represented by an $\epsilon_i$ half-twist about the arc $\gamma'_i$ where $\epsilon_i \in \{\pm1\}$ is the sign of the hyperbolic point that $\gamma_i$ contains. 
Let $w = w_1 \cdots w_m \in MCG(S, P)$. 
Then we see that $F$ is homeomorphic to the surface $F_w$.
\end{proof}

\begin{corollary}
If a closed braid $K$ is $\alpha$-strongly quasipositive with respect to $(S, \phi)$ if and only if $\beta_K$ is a product of positive band-twists.
\end{corollary}

\subsection{Minimum genus Bennequin surfaces}\label{subsec:4.3}

In this section, we prove Theorems \ref{theorem:Q3to12} and \ref{theorem:BEsurface_top}.
We begin with a simple observation that if a transverse link is the boundary of a Bennequin surface then it satisfies the Bennequin-Eliashberg inequality.

\begin{lemma}
\label{lemma:deltacondition}
Let $\mT$ be a null-homologous transverse link and $\alpha \in H_{2}(M,\mT;\Z)$ be a Seifert surface class. If $\mT$ is the boundary of a minimal genus $\alpha$-Bennequin surface  
then $\delta(\mT,\alpha) \geq 0$. 
\end{lemma}

\begin{proof}
Since any $\alpha$-Bennequin surface has $e_{-}=0$, Lemma \ref{lemma:sl-from-ob} gives $\delta(\mT,\alpha)= h_{-} \geq 0$.
\end{proof}

Proposition \ref{prop:newformulation} and Theorem \ref{theorem:Q3to12} easily follow from  Lemma \ref{lemma:sl-from-ob}.

\begin{proof}[Proof of Proposition \ref{prop:newformulation}]
If $(M,\xi)$ is overtwisted, then by the Bennequin-Eliashberg inequality theorem (Theorem~\ref{thmBE}), there is a transverse link $\mT$ and its Seifert surface class $\alpha$ such that $\delta(\mT,\alpha)<0$ (e.g. take a transverse push-off of the boundary of an overtwisted disk). By Lemma \ref{lemma:deltacondition} such a transverse link $\mT$ cannot bound a minimum genus $\alpha$-Bennequin surface.
This proves the contrapositive of the first statement of the proposition. 

To see the second statement of the proposition, we assume that $(M,
\xi)$ is tight. 
Theorem~\ref{thmBE} and the truth of 
Conjecture \ref{conjecture:minimumBEsurface} imply that for any null-homologous transverse link $\mT$ and its Seifert surface class $\alpha$, $\mT$ bounds a minimal genus $\alpha$-Bennequin surface with respect to $(S, \phi)$. 
\end{proof}

\begin{proof}[Proof of Theorem \ref{theorem:Q3to12}]
Assume that $\delta(\mT,\alpha)\geq 0$ for some $\mT$ and $\alpha$. 
The truth of Conjecture~\ref{conjecture:minimumBEsurface} implies that there exists an $\alpha$-Bennequin surface $F$ with $g(F)=g(\mT,\alpha)$.
Let $p$ (resp. $n$) be the number of positively (resp. negatively) twisted bands in $F$. By a property of the geometric definition of an $\alpha$-Bennequin surface, $p$ (resp. $n$) is equal to the number of positive (resp.  negative) hyperbolic points of the open book foliation $\F(F)$. Any $\alpha$-Bennequin surface has $e_{-}=0$. By Lemma \ref{lemma:sl-from-ob} $\delta(\mK,\alpha) = h_-=n.$ Thus Conjectures 1$'$ and 2$'$ hold.  
\end{proof} 

Next we prove Theorem \ref{theorem:BEsurface_top}, which guarantees the existence of minimum genus Bennequin surfaces for every {\em topological} link type.

Let $C$ be a connected component of the binding of the open book $(S, \phi)$.
Let $\mu_{C}$ be a meridian of $C$ whose orientation is induced from the orientation of $C$ by the right hand rule. 
We say that a closed braid $K'$ is a \emph{positive}  (resp. \emph{negative}) {\em stabilization} of a closed braid $K$ about $C$, if $K'$ is the band sum of  $\mu_C$ and $K$ with a positively (resp. negatively) twisted band. See Figure \ref{fig:stabilization} (i).
Here, a positively (resp. negatively) twisted band is an oriented rectangle 
whose foliation induced by the pages of the open book 
has a unique positive (resp. negative) hyperbolic point. 
The boundary edges are oriented by the boundary orientation of the rectangle. The edges $\overrightarrow{bc}$ and $\overrightarrow{da}$ are in braid position; namely, positively transverse to the pages of the open book. On the other hand, the edges $\overrightarrow{ab}$ and $\overrightarrow{cd}$ negatively transverse to the pages. (Therefore, the condition ($\mathcal F$ ii) of open book foliations is not satisfied at the corner points $a, b, c$ and $d$.)  
The edge $\overrightarrow{ab}$ is attached to $\mu_C$ and the edge $\overrightarrow{cd}$ is attached to $K$.   
\begin{figure}[htbp]
\begin{center}
\includegraphics*[
width=130mm]{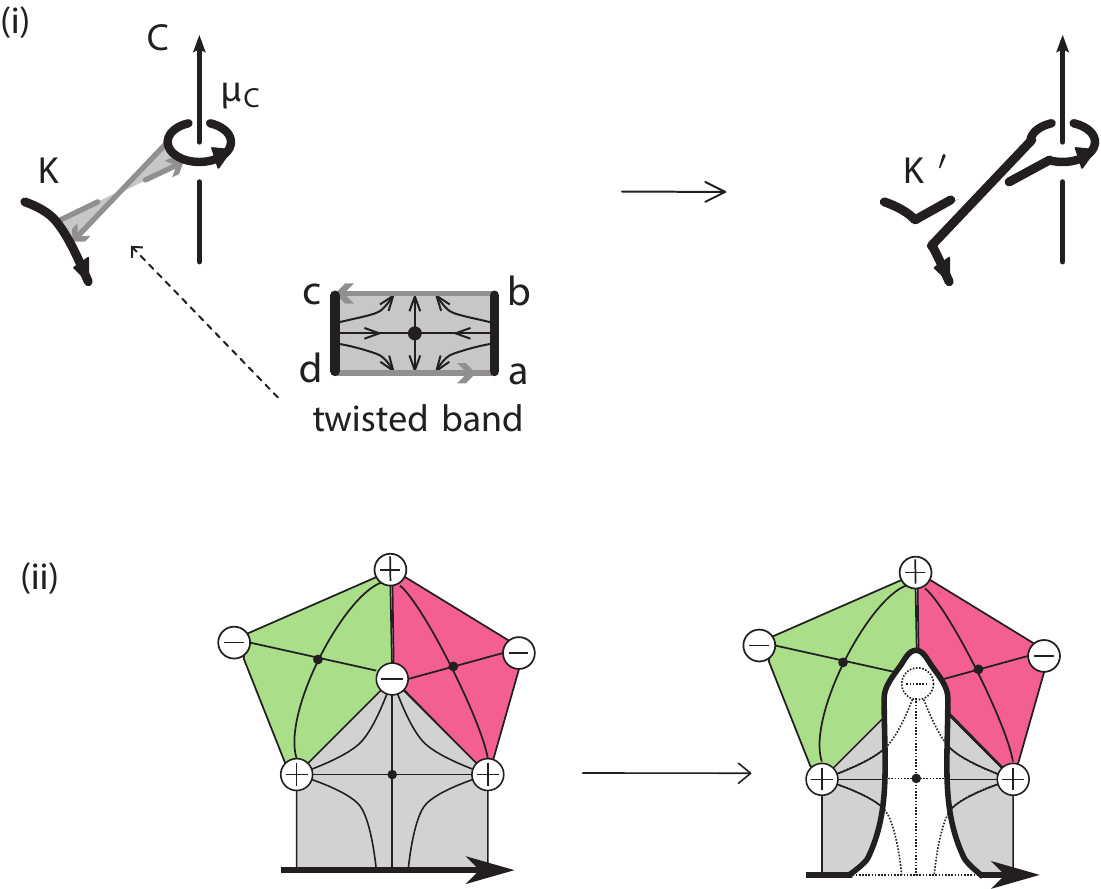}
\caption{(i) Stabilization about the binding component $C$. 
(ii) Stabilization along an ab-tile (gray). The two bb-tiles (green and red) become ab-tiles. }
\label{fig:stabilization}
\end{center}
\end{figure}

Both positive and negative stabilizations preserve the topological link type of the closed braid $K$. 
A positive stabilization preserves the transverse link type of $K$, whereas a negative stabilization does not. Recall the fact that one can remove an ab-tile by a stabilization, see Figure \ref{fig:stabilization} (ii) and \cite[Figure 26]{BM}.

\begin{lemma}
\label{lemma:ab-remove}
Let $K$ be a null-homologous closed braid with respect to $(S, \phi)$ and $F$ be a Seifert surface of $K$ admitting an open book foliation. 
Assume that the region decomposition of $F$ has an ab-tile $R$.
Let $C$ denote the binding component on which the negative elliptic point of $R$ lies. 
If $\sgn(R)=+1$ (resp. $-1$) then a \emph{negative} (resp. {\em positive}) stabilization of $K$ about $C$ can remove the ab-tile $R$. As a result, ab-tiles, bb-tiles and bc-annuli that share the negative elliptic point become aa-tiles, ab-tiles and ac-annuli, respectively.  
\end{lemma}

\begin{proof}
We push $K=\partial F$ across the unstable separatrix of the hyperbolic point in $R$. 
See Figure~\ref{fig:stabilization}~(ii).
Then the resulting closed braid is a stabilization of $K$ about $C$ and the sign of stabilization is positive (resp. negative) if $\sgn(R)=-1$ (resp. $\sgn(R)=$+1). 
\end{proof}

\begin{proof}[Proof of Theorem \ref{theorem:BEsurface_top}]
Take a closed braid representative $K$ of a null-homologous topological link type $\mK$. 
Let $F$ be an $\alpha$-Seifert surface of $K$ with $g(F)=g(\mK,\alpha)$.
By an isotopy fixing $K$ we may put $F$ in a position so that $F$ admits an open book foliation $\F(F)$. 
By \cite[Proposition 2.6]{ik1-1} we may assume that $\F(F)$ contains no c-circles. 

By Lemma \ref{lemma:ab-remove}, after sufficiently many positive and negative stabilizations, we can remove all the ab-tiles without producing new c-circles.
As a consequence, existing bb-tiles may become ab-tiles. Then we remove the new ab-tile as well by further stabilizations.  
After removing all the ab-tiles and bb-tiles, the region decomposition consists of only aa-tiles; thus, we obtain an $\alpha$-Bennequin surface.
\end{proof}

\section{Proofs of the main theorems}\label{sec:5}

The goal of this section is to prove the main results (Theorems~\ref{theorem:main1} and \ref{theorem:main2} and Corollary~\ref{cor:main2}). 

\subsection{Lemmas for the main results}

Let $F$ be a Seifert surface of a closed braid $K$ with respect to $(S,\phi)$.
Assume that $F$ admits an open book foliation $\F(F)$. 
Fix a region decomposition of $\F(F)$. 
To relate the open book foliation and the FDTC, we use the following graph $\widehat{G_{--}}$ which is a slight modification of the graph $G_{--}$ introduced in \cite[Definition 2.17]{ik1-1}.

\begin{definition}\label{def:extendedgraph}
Let $R$ be an ab-tile, a bb-tile or a bc-annulus in the region decomposition of $\F(F)$. 
If $\sgn(R)=-1$ then the graph $G_R$ on $R$ is as illustrated in Figure~\ref{figure:graph}. 
If $\sgn(R)=+1$ then $G_R$ is defined to be empty.
Also, if $R$ is an aa-tile, an ac-annulus or a cc-pants then $G_R$ is defined to be empty. 
The union of graphs $G_R$ over all the regions of the region decomposition and  all the negative elliptic points gives a  (possibly not connected) graph, $\widehat{G_{--}}$, contained in $F$. 
We call the graph $\widehat{G_{--}}$ the \emph{extended graph} of $G_{--}$.

There are two types of vertices in $\widehat{G_{--}}$. 
We say that a vertex of $\widehat{G_{--}}$ is \emph{fake} if it is not a negative elliptic point as depicted with a hollow circle in Figure~\ref{figure:graph}.  
A negative elliptic point is called a {\em non-fake} vertex. 

\end{definition}

\begin{figure}[htbp]
\begin{center}
\includegraphics*[width=80mm]{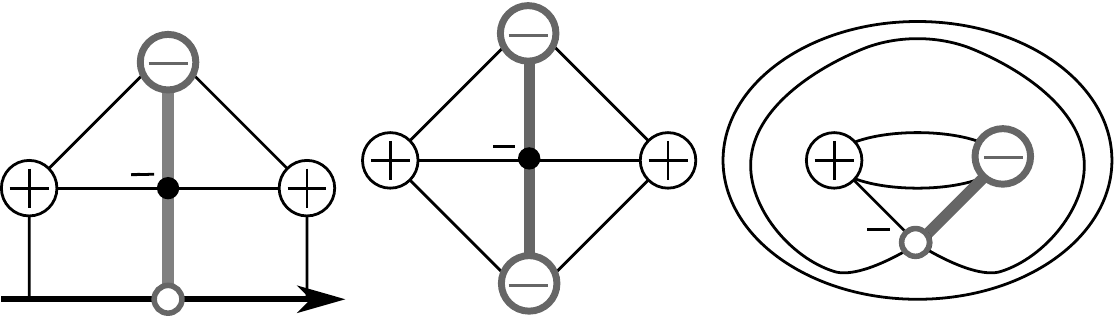}
\caption{The graph $G_R$. 
Edges are thick gray. 
Negative elliptic points are non-fake vertices. 
Hollow circles are fake vertices. 
Black dots are negative hyperbolic points and not vertices. 
An edge may or may not contain a black dot. }
\label{figure:graph}
\end{center}
\end{figure}

In Lemma~\ref{lemma:key} and Propositions~\ref{prop:PreBennequin}, \ref{prop:PreBennequin2} and \ref{prop:noac}, we assume that $K$ is a closed braid with respect to an open book $(S,\phi)$ representing a null-homologous transverse link $\mT$, and $F$ is a Seifert surface of $K$ with   
\begin{enumerate}
\item[(I)]
$\delta(\mT, [F])\geq 0$, 
\item[(II)]
$g(F)=g(K, [F]);$
that is, $F$ is incompressible and $\delta(\mT,[F]) \leq h_{-}-e_{-}$ by Lemma~\ref{lemma:sl-from-ob}, 
\item[(III)] $F$ admits an open book foliation $\F(F)$ (which may be not essential).
\end{enumerate}

\begin{lemma}
\label{lemma:key}
If the open book foliation $\F(F)$ contains negative elliptic points then the extended graph $\widehat{G_{--}}$ contains a non-fake vertex of valence less than or equal to $\delta(\mT,[F])+2$. \end{lemma}

\begin{proof}
First suppose that $e_-=1$. Let $d$ denote the valence of the unique negative elliptic point of $\F(F)$. 
Note that $h_- \geq d$ because every edge of $\widehat{G_{--}}$ contains one negative hyperbolic point. 
By the condition (II), we have 
\begin{equation}\label{eq:5.1}
\delta(\mT,[F]) = h_{-} - e_{-}.  
\end{equation}   
Thus, 
$\delta(\mT,[F]) = h_{-} - e_{-} \geq d-1$ and $d \leq \delta(\mT,[F])+1$. 

Next we assume that $e_- \geq 2$.
For $i\geq 0$, let $v_i$ be the number of vertices of $\widehat{G_{--}}$ whose valence is $i$ and let $w$ be the number of edges of $\widehat{G_{--}}$.
Then we have $$\sum_{i\geq 0}iv_i = 2w$$ and $$\chi=\sum_{i\geq 0}v_{i}-w,$$ where  $\chi=\chi(\widehat{G_{--}})$ is the Euler characteristic of the extended graph $\widehat{G_{--}}$. Therefore, 
\begin{equation}
\label{eqn:euler0} \sum_{i\geq 2}(i-2)v_i = -2\chi + v_1 +2v_0. 
\end{equation}

If there is a non-fake vertex of valence less than or equal to two, then we are done since $2 \leq \delta(\mT,[F])+2$ by the condition (I).

Suppose that every non-fake vertex has valence grater than two. (i.e., $v_{0}=v_{2}=0$). 
Then, since every fake vertex has valence one, $v_1$ is equal to the number of fake vertices.
By Definition~\ref{def:extendedgraph} we have $h_{-}\geq w$ and $$e_{-} = \sum_{i\geq 0}v_{i} - \#\{\mbox{fake vertices}\} = \sum_{i\geq 0}v_{i} - v_1.$$ 
By (\ref{eq:5.1})
\[ 
\delta(\mT, [F]) = h_{-} - e_{-} \geq w-e_{-} =w - (\sum_{i\geq 0}v_i) + v_1 = -\chi +v_1.
\]
Therefore by (\ref{eqn:euler0}) we get an inequality
\begin{equation}
\label{eqn:euler}
\sum_{i\geq 2}(i-2)v_i \leq  2\delta(\mT, [F]) - v_1. 
\end{equation}
Recall that $v_0=v_2=0$. 
Let $j=\min\{i>2 \: | \: v_i \neq 0\}$. Since $\widehat{G_{--}}$ contains at least two non-fake vertices, by (\ref{eqn:euler})
\[ 
(j-2) \cdot 2 \leq 
(j-2) \cdot e_- \leq 
(j-2) \sum_{i\geq 3} v_i \leq
\sum_{i\geq 2}(i-2)v_i \leq  2\delta(\mT, [F]) - v_1. 
\]   
Therefore, $$j \leq \delta(\mT,[F])-\frac{v_1}{2} +2 \leq \delta(\mT,[F])+2. $$
\end{proof}

Propositions \ref{prop:PreBennequin} and \ref{prop:PreBennequin2} below show that under an assumption of large FDTC and essentiality of the open book foliation, an $\alpha$-Seifert surface is `close' to an $\alpha$-Bennequin surface in the sense that its open book foliation has no negative elliptic points (but may have c-circles).

\begin{proposition}
\label{prop:PreBennequin}
Assume that the open book foliation $\F(F)$ is essential.
If $c(\phi,K,C)>\delta(\mT,[F])+2$ for every binding component $C$ that intersects $F$, then $\F(F)$ has no negative elliptic points (but possibly it has c-circles).
\end{proposition}

\begin{proof}
Assume to the contrary that the braid foliation $\F(F)$ has negative elliptic points. 
By Lemma \ref{lemma:key}, there exists a non-fake vertex $v$ of $\widehat{G_{--}}$ whose valence is less than or equal to $\delta(\mT,[F])+2$. By Theorem \ref{lemma:FDTC}, this implies that $c(\phi,K,C) \leq \delta(\mT,[F])+2$, which contradicts our assumption.
\end{proof}

If $F$ intersects exactly one binding component then 
we can say more with a smaller lower bound on the FDTC.

\begin{proposition}
\label{prop:PreBennequin2} 
Assume that the open book foliation $\F(F)$ is essential. 
Let $C$ be a binding component. Let $e_C$ denote the number of negative elliptic points in $\F(F)$ that are on the binding component $C$.
\begin{enumerate}
\item
If $c(\phi,K,C)> \frac{1}{k}\delta(\mT,[F]) + 1$ for some $k\geq 1$ and $\delta(\mT,[F]) >0$
then $e_C < k$.

\item
If $c(\phi,K,C)> 1$ and $\delta(\mT,[F])=0$ then $e_C=0$. 

\item  
If $c(\phi,K,C)> 1$ for {\em all} the binding components and $\delta(\mT,[F])=0$ then $e_-=h_-=0$.  
\end{enumerate}
\end{proposition}

\begin{proof}
If $e_- = 0$ then we are done. 

We may assume that $e_- \geq 1$.  
Let $e_i$ denote the number of {\em negative} elliptic points that are on the binding component $C_i$. We have $e_-=\sum_i e_i$. 

Suppose that $e_i\geq 1$. 
Let $N (\geq 0)$ be the number of negative hyperbolic points of type either ab, bb or bc.
Note that $h_- \geq N$. 
Every negative hyperbolic point of type ab, bb or bc is connected to at least one negative elliptic point by a singular leaf.
By Theorem \ref{lemma:FDTC} and the condition (II) we have 
\begin{equation}\label{eq:inequality}
c(\phi,K,C_i) \leq \frac{N}{e_i} \leq \frac{h_-}{e_i}= \frac{\delta(\mT,[F])}{e_i}  + \frac{e_-}{e_i} \leq \frac{\delta(\mT,[F])}{e_i}  + 1.
\end{equation}

\begin{enumerate}
\item
Assume that 
\begin{equation}\label{eq:inequality2}
\frac{1}{k}\delta(\mT,[F]) + 1<c(\phi,K,C)
\end{equation} 
for some $k\geq 1$ and $\delta(\mT,[F]) >0$. 

If $k=1$ and $e_i\geq 1$ then inequalities (\ref{eq:inequality}) and (\ref{eq:inequality2}) yield $e_i < 1$ which is a contradiction. 
Therefore, when $k=1$ we must have $e_i=0$. 

If $k\geq 2$ and $e_i\geq 1$ then inequalities (\ref{eq:inequality}) and (\ref{eq:inequality2}) yield $e_i < k$. 

\item
If $\delta(\mT,[F])=0$ and $1< c(\phi,K,C_i)$ then inequality (\ref{eq:inequality})  gives $1 < 1$, which is a contradiction. 
Therefore in this case $e_i=0$. 

\item
The last statement follows from (2) and $e_-=\sum_i e_i$. 
\end{enumerate}
\end{proof}

The next proposition gives a criterion of strongly quasipositive braids. 

\begin{proposition}
\label{prop:noac}
Assume the following.  
\begin{enumerate}
\item[(i)] All the elliptic and hyperbolic points of $\F(F)$ are positive. 
\item[(ii)] The page $S$ is planar.
\item[(iii)] Only one binding component intersects $F$. 
\end{enumerate}
Then $F$ is an $[F]$-Bennequin surface and $K$ is a strongly quasipositive braid. 
\end{proposition}

\begin{proof}
Let $C$ be the unique binding component that intersects $F$. 
By the assumption (ii), if there exists a c-circle, $c,$ in a page $S=S_t$ then $c$ separates $S$ into two components. 
Let $X$ be the connected component of $S \setminus c$ that contains $C$. 

Recall our orientation convention for leaves as defined in Section~\ref{section:OBFsummary}.  
We say that $c$ is \emph{coherent} with respect to $C$ if the leaf orientation of $c$ agrees with the boundary orientation of $c \subset \partial X$. Otherwise, we say that $c$ is \emph{incoherent}. 
For simplicity, we omit writing `with respect to $C$' in the following.

By the assumption (i), there are no negative elliptic points. 
Therefore, 
the region decomposition of $F$ consists of only aa-tiles, ac-annuli and cc-pants each of which has a positive hyperbolic point.

First, let us consider 
how an ac-singular point changes the types of local regular leaves.

\noindent
(1) 
An a-arc forms a positive hyperbolic point $h$ with itself then splits into an a-arc and a c-circle, $c$, see Figure~\ref{figure:coherent} (1). 
By the assumption (iii), every a-arc starts at $C$. 
This shows that the c-circle $c$ must be incoherent. 

\noindent
(2) 
An a-arc and a c-circle merge and form a positive hyperbolic point $h$. Then they become one a-arc, see Figure~\ref{figure:coherent} (2).  
By the assumption (iii), this c-circle must be coherent.

Next, let us consider how a cc-hyperbolic point changes the types of local regular leaves.

\noindent
(3)
Suppose that a c-circle forms a positive cc-hyperbolic point $h$ with itself then splits into two c-circles. 
There are two possibilities. 

\noindent
(3-a) 
An incoherent c-circle splits into two incoherent c-circles, see Figure~\ref{figure:coherent} (3-a).  

\noindent
(3-b)
A coherent c-circle splits into one coherent c-circle and one incoherent c-circle, see Figure~\ref{figure:coherent} (3-b).  

\noindent(4) 
Suppose that two c-circles merge and form a positive cc-hyperbolic point then become one c-circle. There are two possibilities. 

\noindent
(4-a) Two coherent c-circles merge into one coherent c-circle, see Figure~\ref{figure:coherent} (4-a).

\noindent
(4-b) One coherent c-circle and one incoherent c-circle merge 
into one incoherent c-circle, see Figure~\ref{figure:coherent} (4-b).  
\begin{figure}[htbp]
 \begin{center}
\includegraphics*[width=120mm, 
bb=111 189 468 721]
{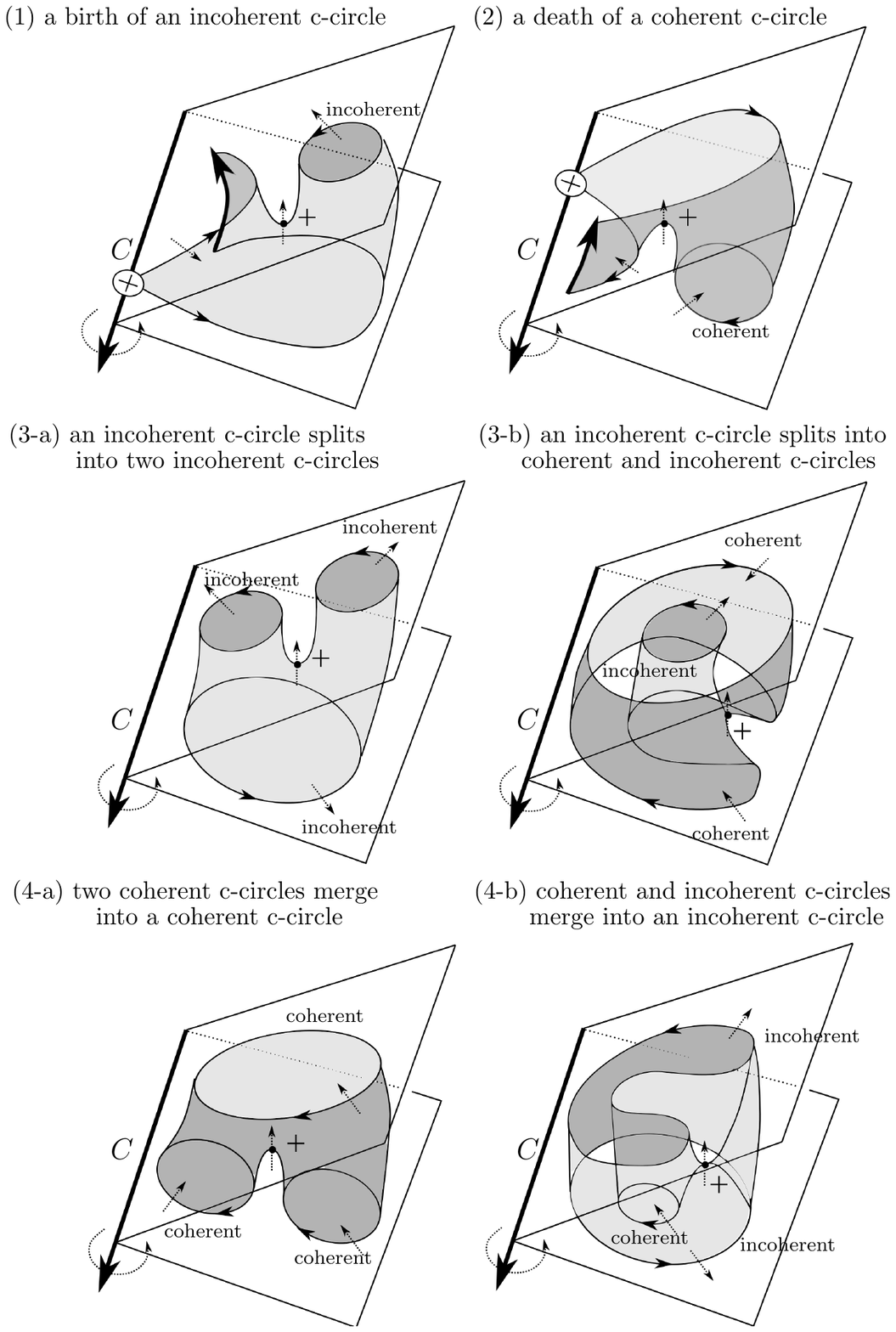}
\caption{
Coherent and incoherent c-circles near positive hyperbolic points. 
The positive (negative) side of the oriented surface $F$ is colored light (dark) gray. 
Dashed arrows are positive normal vector $\vec{n}_F \in T_p S_t \subset T_pM$ to the surface $F$. }
\label{figure:coherent}
  \end{center}
\end{figure}

The above discussion shows that passing a type ac or cc positive hyperbolic point never decreases (resp. increases) the number of incoherent (resp. coherent) c-circles. 
For a regular page $S_t$ $(t \in [0,1])$, let $N(t)$ be the number of incoherent c-circles in $S_t$. 
Since a type aa hyperbolic point does not affect c-circles we get $N(t) \leq N(t')$ for $t <t'$.

Our strategy is to show that all the regions in the region decomposition are of type aa; hence, $F$ is an  $[F]$-Bennequin surface. 

Assume to the contrary that there exist c-circles. 
If no a-arcs interact with those c-circles (i.e., no ac-annuli exist), then the surface $F$ contains a component consisting of only aa-tiles. 
In other words, $F$ is disconnected, which is a contradiction. 

Therefore, $\F(F)$ contains ac-annuli.

If at least one ac-annuls of type (1)  exists then we have an strict inequality $N(0) < N(1)$.
However, the page $S_1$ is identified with the page $S_0$ by the monodromy $\phi$ of the open book. Since $F$ is orientable, $\phi$ identifies an incoherent c-circle in $S_1$ with an incoherent c-circle in $S_0$, which means $N(0)=N(1)$. This is a contradiction.

If $\F(F)$ contains an ac-annulus of type (2) then a parallel argument about the number of coherent c-circles holds and we get a contradiction. 

Thus, c-circles do not exist.
\end{proof}

Careful readers may notice that the conditions (I) and (II) are not used in the proof.

\subsection{Proofs of the main results}

\begin{proof}[Proof of Theorem \ref{theorem:main1}]
($\Leftarrow$) The statement is trivial. 

($\Rightarrow$) We assume $\delta(\mT,\alpha)=0$ and show that $K$ is an $\alpha$-strongly quasipositive braid.

By the assumptions (ii) and (iii-a) and Corollary \ref{cor:essential}, after  desumming essential spheres, we may assume that the new $F$ (we abuse the same notation) admits an essential open book foliation $\F(F)$.
Here desumming an essential sphere can be understood as a disk exchange: We remove  an embedded disk $D \subset F$ from $F$ and then put back a disk $E \subset M \setminus F$ such that $E \cup D$ is an essential sphere $\mathcal{S}$. 
Corollary \ref{cor:essential} states that the binding component $C$ in the condition (iii-b) is still the only binding component that intersects the new $F$.

Applying the assumptions $\delta(\mT,\alpha)=0$ and (iii-c) to Proposition \ref{prop:PreBennequin2}, we can conclude that $\F(F)$ has $e_- =h_-= 0$. 
By Proposition \ref{prop:noac} and the assumptions (i) and (iii-b), $F$ is an $\alpha$-Bennequin surface with the strongly quasipositive boundary $K$.
\end{proof}

\begin{proof}[Proof of Theorem \ref{theorem:main2}]
Let $F$ be a minimum genus Seifert surface of $K$.
By Corollary~\ref{cor:essential}, we may assume that $F$ admits an essential open book foliation.

In the case of disk open book $(S, \phi)=(D^2, id)$,  
\cite[Lemma 2]{bm1} and \cite[Lemmas 1.2 and 1.3]{bf} show  that any incompressible surface can be put in a position so that its open book foliation is essential without c-circles. 
Therefore, our Seifert surface $F$ admits an essential open book foliation without c-circles.  
This means that $F$ contains no bc-tiles and all the fake vertices (if they exist) of $\widehat{G_{--}}$ lie on $K$.

Since $c(id,K, \partial D^{2}) > \frac{\delta(\mT)}{2}+1$, 
Proposition \ref{prop:PreBennequin2} implies that $\F(F)$ has at most one negative elliptic point; $e_-=0$ or $1$. 

If $e_-=0$ then the region decomposition of $\F(F)$ consists of only aa-tiles; that is, $F$ is a Bennequin surface.

If $e_-=1$ then $\F(F)$ does not contain bb-tiles. 
Let $v$ denote the unique negative elliptic point. 
All the ab-tiles of $\F(F)$ meet at $v$. 
Suppose that the valence of $v$ in the graph $\widehat{G_{--}}$ is $N$. 
By Theorem \ref{lemma:FDTC} we have 
$$\frac{\delta(\mT)}{2}+1 < c(id,K, \partial D^{2}) \leq \frac{N}{1}=N,$$ 
which shows that there are $N \geq 2$ negative ab-tiles meeting at $v$.
By applying a positive stabilization along one of the negative ab-tiles (cf. Figure~\ref{fig:stabilization} (ii)) we may remove the negative elliptic point $v$. 
Note that the genus of the surface is preserved. 
As a consequence we get a Seifert surface whose region decomposition consists of only aa-tiles.  
Hence by Definition~\ref{definition:BEsurface_geom}-(1) it is a Bennequin surface.

Moreover, if $\delta(\mT)=0$ and $c(id, K, \partial D^2) >1$ then by Proposition \ref{prop:PreBennequin2}-(3) we have $e_- =h_-= 0$.  
The Seifert surface $F$ is already a Bennequin surface without negatively twisted bands; hence, by Definition~\ref{definition:BEsurface_geom}-(2) $K$ is strongly quasipositive.
\end{proof}

\subsection{Examples}

We close the paper with examples related to the main results.  
Some of the examples are described via \emph{movie presentations}. 
A movie presentation is a sequence of slices of a Seifert surface by some pages $S_{t}$. See \cite[p.1597]{ik1-1} for the definition of a movie presentation.

\begin{example}
\label{ex1}
First we see that the planar condition (i) of Theorem \ref{theorem:main1} is necessary.

Suppose that $S$ is an oriented genus $1$ surface with connected boundary. 
Choose $\phi$ so that the the manifold $M_{(S, \phi)}$ is a rational homology sphere. The condition (ii) of Theorem \ref{theorem:main1} is automatically satisfied. Since the Seifert surface class is uniquely determined we may drop $\alpha$- from our notation.

Take a base point near the boundary so that $\phi$ fixes it.
Let $\rho$ be an oriented loop at this base point as depicted in Figure~\ref{fig:rho} (1). 
Under the identification $B_{1}(S) = \pi_{1}(S)$ we may identify $\rho$ with a 1-braid in the surface braid group $B_1(S)$. 
\begin{figure}[htbp]
 \begin{center}
\includegraphics*[width=85mm,bb=176 379 434 715]
{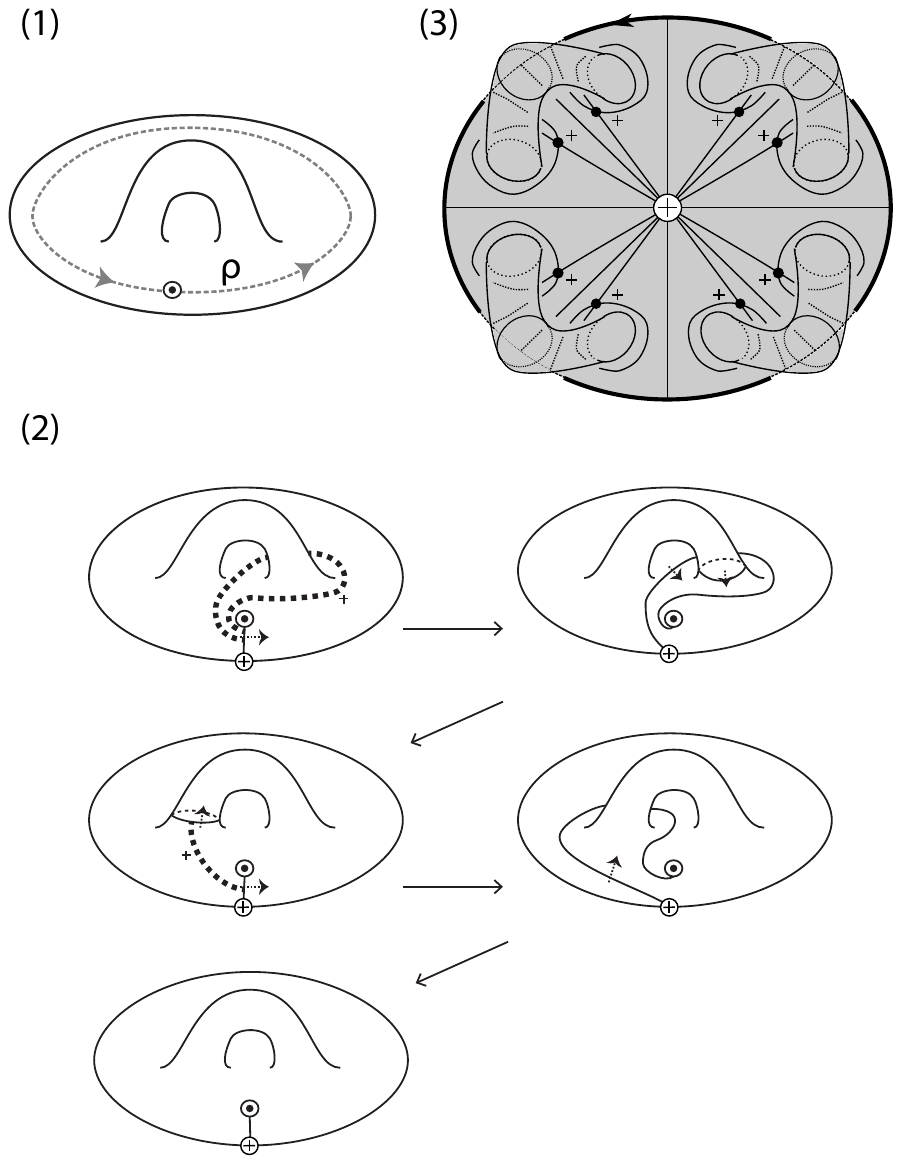}
\caption{(1): The oriented loop $\rho$. 
(2) A movie presentation of the Seifert surface $F$ for the closed braid $K_N$ where $N=1$. For $N\geq 2$ iterating the movie $N$ times gives the surface $F$. 
(3) A open book foliation of $F$ where $N=4$.}
\label{fig:rho}
  \end{center}
\end{figure}

For $N\geq1$, let $K_{N}$ be the closure of the 1-braid $\rho^{N}$ with respect to the open book $(S,\phi)$. Then $c(\phi,K_N,\partial S)=N$.  The condition (iii-c) is satisfied if $N>1$. 

Note that $K_1$ is smoothly isotopic to the binding of the open book. 
This shows that $g(K_1)= g(S)=1$. Since $K_N$ is an $(N,1)$-cable of the binding, $K_N$ 
is a connected sum of $N$ copies of $K_1$, which yields $g(K_N)=N$. 
The Seifert surface $F$ of $K_N$ defined by 
the movie presentation in Figure~\ref{fig:rho} (2) gives a genus $N$ surface. Therefore, the condition (iii-a) is satisfied.

The movie presentation also determines the open book foliation $\F(F)$ of $F$ as depicted in Figure~\ref{fig:rho} (3). 
We observe that $\F(F)$ is essential (there are no b-arcs) and all the hyperbolic and elliptic points are positive. 
By Lemma~\ref{lemma:sl-from-ob} it follows that $\delta(K_{N})=0$.

If $K_N$ were a strongly quasipositive braid bounding a Bennequin surface $F'$ then 
due to the one-strand constraint 
the open book foliation $\F(F')$ must be built of only a-arcs emanating from a single positive elliptic point. 
This means that $K_N$ is a meridional circle of the binding; that is, an unknot. This contradicts the above conclusion $g(K_N)=N\neq 0$.

We conclude that 
if $N>1$, all the conditions of Theorem~\ref{theorem:main1} are satisfied except for the planar assumption (i) on $S$, and $K_N$ is not a strongly quasipositive braid. 
\end{example}

\begin{example}\label{example2}
With the above example we can further see that the stronger forms of Conjectures~\ref{conjecture:generalSQP} and \ref{conjecture:SQP} do not hold for general open books $(S, \phi) \neq (D^2, id)$.

More concretely, we show that the transverse knot type $\mT$ represented by the closed braid $K_{1}$ in Example~\ref{ex1} does bound a minimum genus Bennequin surface (indeed, is strongly quasipositive) at the cost of raising the braid index. 

To see this, we consider a different Seifert surface $F'$ of $K_1$ given by a movie presentation as depicted in Figure~ \ref{fig:example2}. 
Using the Euler characteristic formula in Lemma~\ref{lemma:sl-from-ob} we see that both $F$ and $F'$ have genus $1$, which is the genus of the transverse knot type $\mT$. 
However the open book foliations of $F$ and $F'$  are different. 
For instance, the region decomposition of the open book foliations $\F(F)$ consists of two ac-annuli whereas $\F(F')$ consists of four ab-tiles. 
More precisely, $\F(F')$ contains one negative elliptic point and one negative hyperbolic point and they belong to the same unique negative ab-tile. 
\begin{figure}[htbp]
\begin{center}
\includegraphics*[width=85mm]{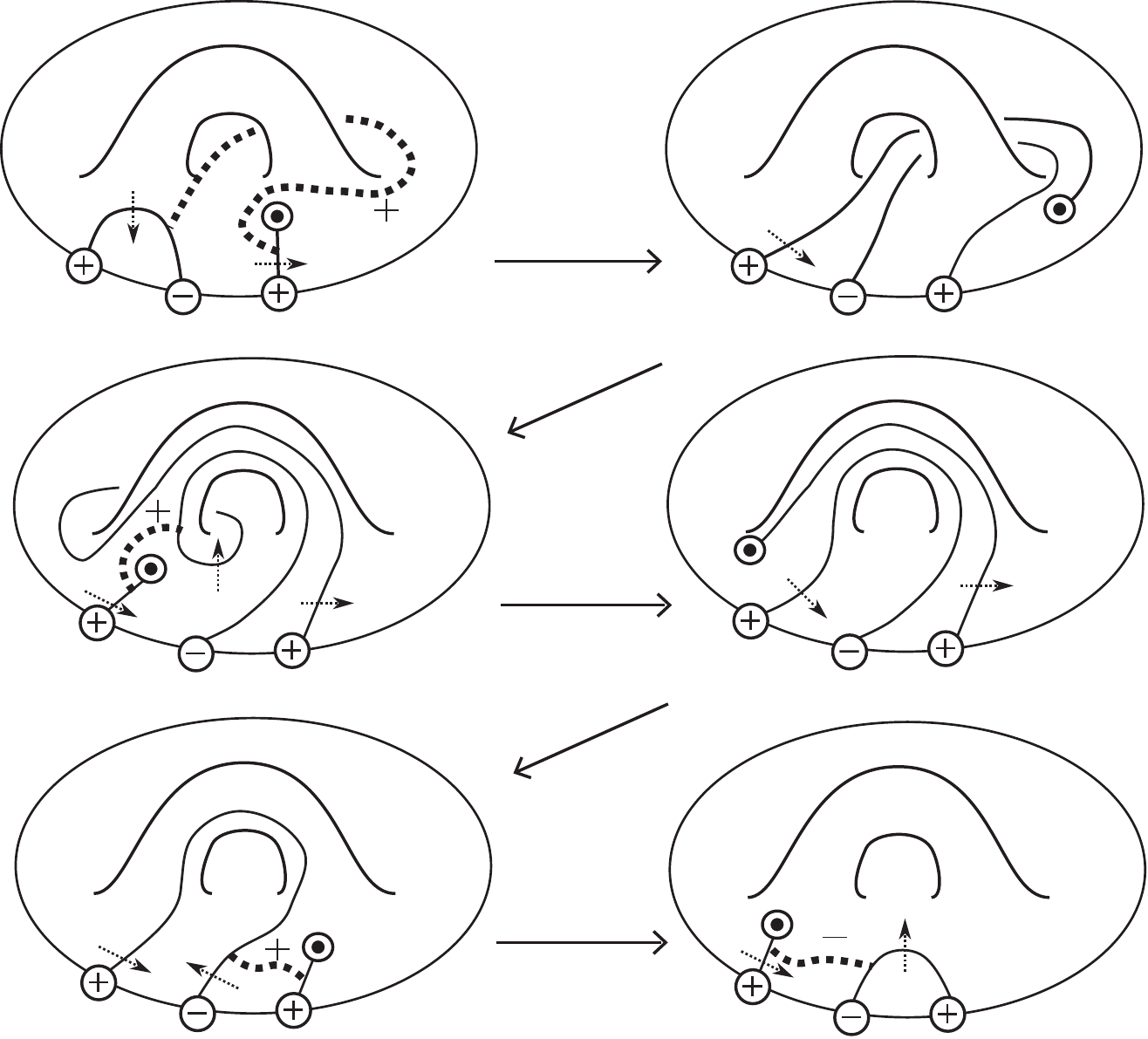}
\caption{A movie presentation of the Seifert surface $F'$ for $K_1$.}
\label{fig:example2}
\end{center}
\end{figure}

By a positive stabilization along the negative ab-tile we can remove both the negative elliptic and negative hyperbolic points of $F'$.
Since any stabilization preserves the Euler characteristic of the surface, the resulting surface, $F''$, also has genus 1. 
The surface $F''$ consists of only positive aa-tiles and its boundary is a closed braid of braid index 2. 

In summary, we obtain a minimum genus Bennequin surface $F''$ of $\mT$ whose boundary is a strongly quasipositive $2$-braid.
Knowing that the braid index $b(\mT)=1$ and $\mT$ is not an unknot, any closed 1-braid representatives of $\mT$ are not strongly quasipositive.

For the closed braid $K_N$ of $N\geq 2$, we add extra $N-1$ pairs of positive and negative elliptic points as shown in Figure~\ref{fig:example2add}. 
The movie presentation in Figure~\ref{fig:example2add} gives a Seifert surface $F'_{N}$ for $K_N$ and $F'_N$ has genus $N$. 
A parallel argument for $F'=F'_1$ works for general $F'_N$ and we obtain the same conclusion.  
\begin{figure}[htbp]
 \begin{center}
\includegraphics*[width=90mm,bb=148 367 462 714]{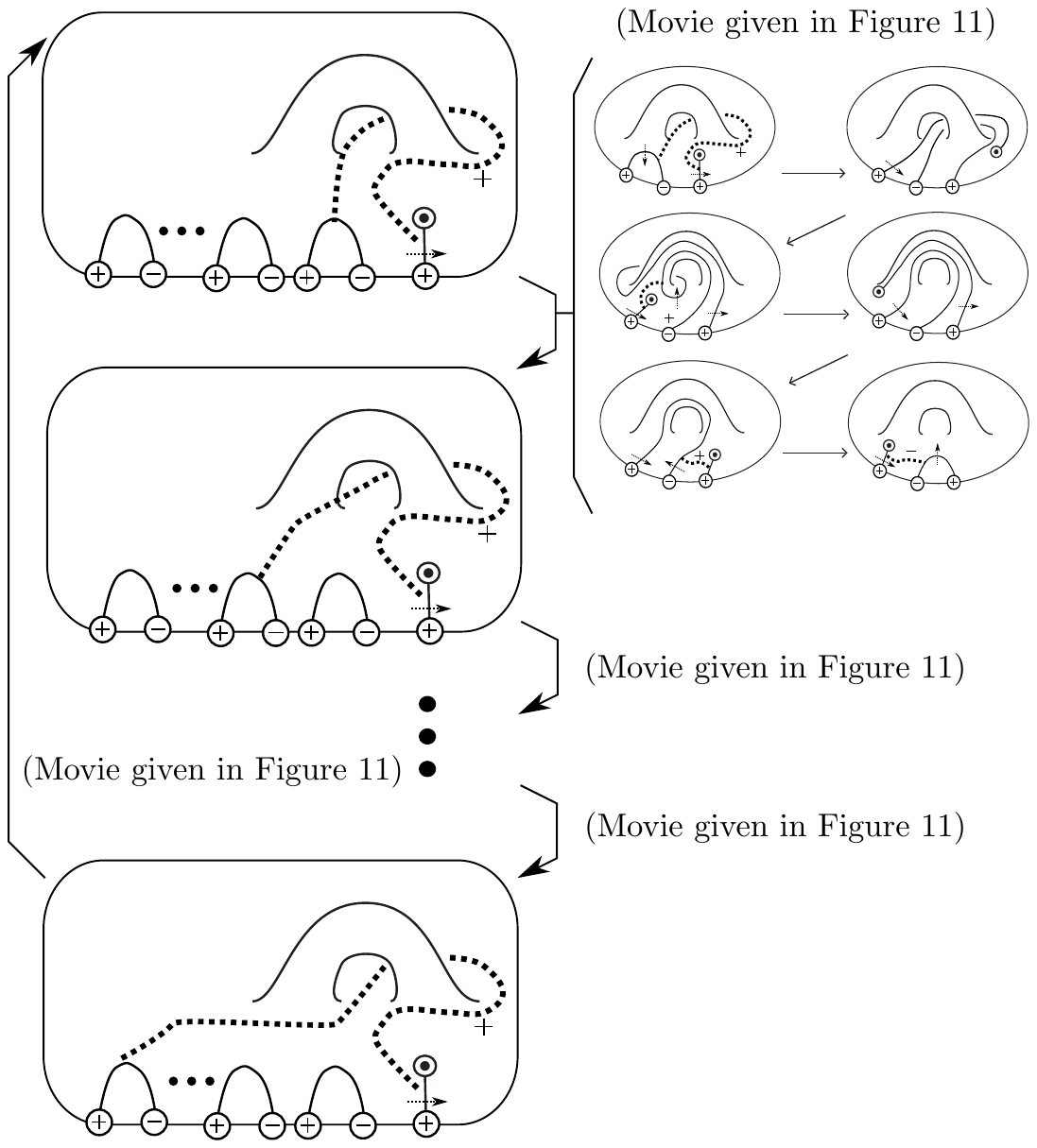}
\caption{A movie presentation of the Seifert surface $F'_N$ for $K_N$.}
\label{fig:example2add}
  \end{center}
\end{figure}
\end{example}

\begin{example}\label{ex2}
Next we see that the  condition (iii-c) on the FDTC in Theorem~\ref{theorem:main1} is also necessary.

Let $S$ be a genus 0 surface with four boundary components 
$C_0,C_1,C_2,C_3$. 
Let $X$ be a simple closed curve that separates $C_1$ and $C_2$ from $C_3$ and $C_4$. 
See Figure~\ref{fig:cexample1} (1).
Let $\phi \in {\rm Diffeo}^+(S)$ be a diffeomorphism defined by 
$$\phi = T_{X}T_{C_1}^{n_1}T_{C_2}^{n_2}T_{C_3}^{n_3},$$ 
where $T_X$ (resp. $T_{C_i}$) denotes a positive Dehn twist about $X$ (a circle parallel to $C_i$) and $n_1,n_2,n_3 \in \Z\setminus \{0\}$.  
Since  $n_1,n_2,n_3 \neq 0$ the ambient manifold $M=M_{(S, \phi)}$ has $H_1(M;\Z)=\Z/n_1\Z \oplus \Z/n_2\Z\oplus \Z/n_3\Z$ (cf. \cite[p.3136]{EO}) and it yields $H_2(M; {\mathbb Q})=0$ by the universal coefficient theorem; hence, $M$ is a rational homology sphere and the condition (ii) of Theorem \ref{theorem:main1} is automatically satisfied.

The movie presentation shown in Figure~\ref{fig:cexample1} (3) gives a surface, which we call $D$.
The trace of the point $\odot$ gives the boundary, $K$, of $D$. 
In particular, $K$ is a 1-braid with respect to $(S, \phi)$. 
The open book foliation of $D$ as depicted in Figure~\ref{fig:cexample1} (2) shows that 
\begin{itemize}
\item $D$ is a disk and $K$ is an unknot ((iii-a) is satisfied). 
\item Among all the binding components of $(S,\phi)$, only $C_0$ intersects $D$ ((iii-b) is satisfied). 
\item $c(\phi,K,C_0)=0$, which is obtained by noticing that 
the arc $\gamma$ in Figure \ref{fig:cexample1} (1) is fixed by $\phi_K$ and then 
applying \cite[Lemma 4.1]{ik2}.
\end{itemize}
Therefore, all the conditions of Theorem \ref{theorem:main1} are satisfied except for the condition (iii-c) on the FDTC. 
Indeed, the region decomposition of $D$ consists of two ab-tiles and $D$ is not even a Bennequin surface; thus $K$ is not a strongly quasipositive braid. 
(We remark that after one positive stabilization, we get a strongly quasipositive braid representative of the transverse knot type $[K]$).

\begin{figure}[htbp]
 \begin{center}
\includegraphics*[width=100mm, bb=148 482 462 717]{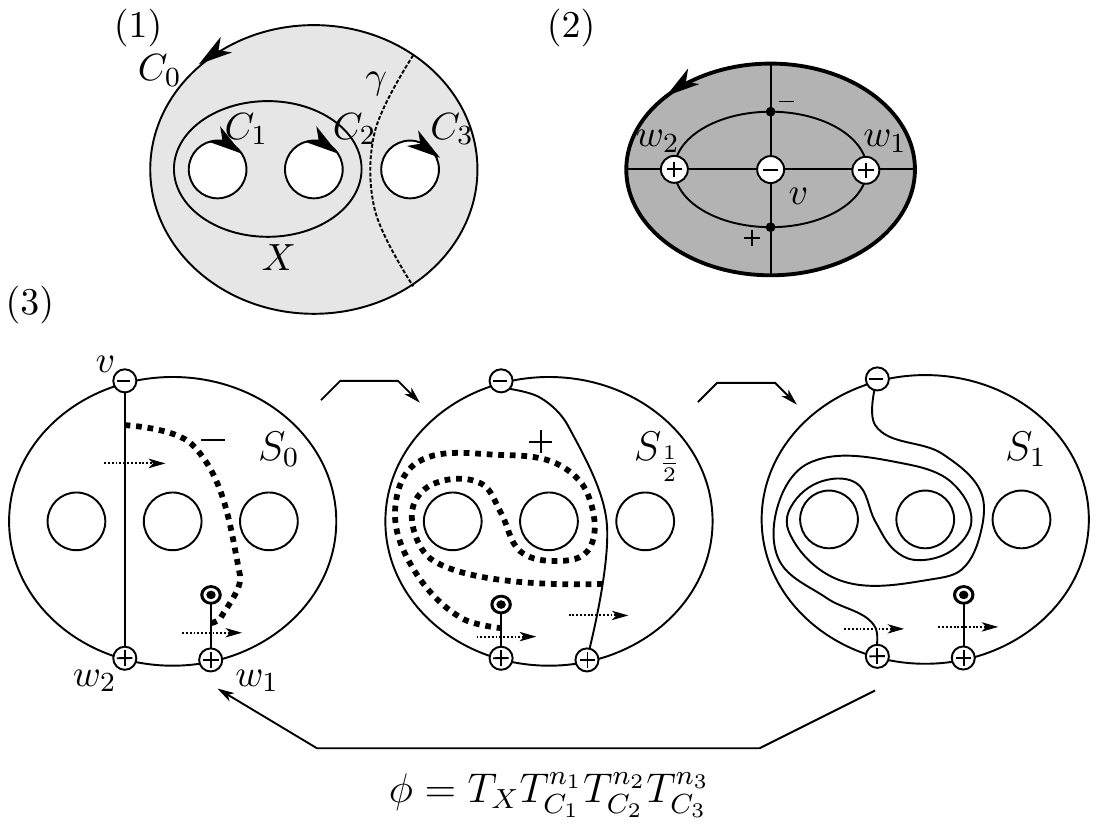}
\caption{(1) The page surface $S$. (2) The open book foliation of $D$. (3) A movie presentation of $D$.}
\label{fig:cexample1}
  \end{center}
\end{figure}
\end{example}

\begin{example}[Example for Theorem~\ref{theorem:main2}]
\label{example:main}

We demonstrate that the condition  $c(id,K,\partial D^{2}) > \frac{\delta(\mT)}{2}+1$ in Theorem~\ref{theorem:main2} can be satisfied by links that are neither 3-braid links or fibered knots. 
That is, Theorem~\ref{theorem:main2} is independent of Corollary~\ref{cor:3braids} and Hedden's \cite[Theorem 1.2]{he}.

For a non-negative integer $\delta \geq 0$, let us consider an $n$-braid word of the form $w=xy$ where $x \in B_n$ is a strongly quasipositive braid word and $y\in B_n$ is a braid word containing $\delta$ negative band generators.

Let $K$ be the closure of $w$ and $\mT$ be the transverse knot type represented by $K$. 
Let $e_{\pm}$ (resp. $h_\pm$) denotes the number of $\pm$ elliptic (resp. hyperbolic) points in the open book foliation $\F(F_w)$ of the Bennequin surface $F_{w}$ associated to the band-twist factorization $w$. 
By Lemma~\ref{lemma:sl-from-ob} we get  $\delta(\mT)\leq  h_{-} - e_-$. 
The proof of Proposition~\ref{prop:4.5} shows that $h_-=\delta$ and $e_-=0$; therefore, $$\delta(\mT)\leq \delta.$$

Let $c:B_n \rightarrow \Q$ be the FDTC map defined by $c(\beta):=c(id,\widehat{\beta},\partial D^{2})$. The map has the following properties for $\alpha, \beta \in B_n$. 
\begin{itemize}
\item[(i)] $|c(\alpha \beta)-c(\alpha)-c(\beta)| \leq 1$ and $c(\alpha)=c(\beta^{-1}\alpha \beta)$. \item[(ii)] If $p\in B_n$ is a strongly quasipositive word then 
$c(\alpha p \beta) \geq c(\alpha \beta) \geq c(\alpha p^{-1} \beta)$ (indeed this holds for right-veering braids $p$ \cite[Corollary 3.1]{ik4}). 
\item[(iii)] $c(\sigma_{i,j}^{\pm 1}) = 0$.
\item[(iv)] If $\beta$ is the product of $m$ negative band generators then  $c(\beta) > -\frac{m+1}{n}$. 
\end{itemize}
Property (i) can be found in \cite{M} and \cite[Corollary 4.17]{ik2}.
Property (iii) follows from \cite[Lemma 4.13]{ik2}. Property (iv) follows from the proof of \cite[Proposition 2.4]{it0}, which is an estimate of another invariant of braids called the Dehornoy floor $[\beta]_{D}$, together with \cite[Lemma 4.13]{ik2}.

By (ii) and (iv) we have $c(y) > - \frac{\delta+1}{n}$. 

Now let us take a strongly quasipositive braid word $x$ such that
$c(x) \geq \frac{\delta}{2} + \frac{\delta+1}{n} +2$: 
For example, 
$x=(\sigma_{1,2}\sigma_{2,3} \cdots\sigma_{n-1,n}\sigma_{1,n})^{N}$ for $N \geq \frac{1}{2}\delta + \frac{\delta+1}{n} + 2$ satisfies this condition (see Remark \ref{remark:dualGarside}). 

By (i) we have
\[ c(w)\geq c(x)+c(y)-1 > \left( \frac{\delta}{2} + \frac{\delta+1}{n} +2 \right) - \frac{\delta+1}{n}-1 = \frac{\delta}{2}+1 \geq \frac{\delta(\mT)}{2}+1.
\] 
The closed braid $K$ satisfies the conditions in Theorem~\ref{theorem:main2} and Corollary~\ref{cor:main2}.
Thus, $\mT$ 
admits a minimum genus Bennequin surface with exactly $\delta(\mT)$ negative bands, even though the Bennequin surface $F_{w}$ may not have the minimum genus $g(\mT)$.

For suitable choices of $x$ and $y$, we can easily make $\mT$ non-fibered:

Let $x= (\sigma_{1,3}\sigma_{2,4}\sigma_{1,3}\sigma_{2,4})^{N+1}\sigma_{1,3} \in B_4$ for $N \geq \frac{3\delta+9}{4}$ and $y \in B_4$ be a braid word in $\{\sigma_{1,3}^{\pm 1}, \sigma_{2,4}^{\pm 1}\}$ containing $\delta$ negative band generators.  The closure $K$ of the 4-braid $w=xy$ realizes the braid index $b(\mT)$ of $\mT$. Since the Bennequin surface $F_{w}$ is not connected, the Alexander polynomial of $\mT$ is zero (see \cite[Proposition 6.14]{L}).
In particular, $\mT$ is not fibered.
Using  \cite[Lemma 4.13]{ik2} we obtain $c(x)\geq N$. 
The above argument shows that $K$ satisfies the assumptions of Theorem~\ref{theorem:main2} and Corollary~\ref{cor:main2}.

\end{example}

\section*{Acknowledgements}
The authors would like to thank Inanc Baykur, John Etnyre, Matthew Hedden, Jeremy Van Horn-Morris, and the referee.  
TI was partially supported by JSPS Grant-in-Aid for Young Scientists (B) 15K17540.
KK was partially supported by NSF grant DMS-1206770 and Simons Foundation Collaboration Grants for Mathematicians.

\end{document}